\documentclass[12pt,a4paper,twoside]{article}

\usepackage{fancyhdr,graphicx,subfigure,psfrag}
\usepackage{amsmath,amsfonts,amsthm}
\usepackage[abs]{overpic}
\usepackage{epsfig}

\newcommand{\Lie}{\mathcal{L}}
\newcommand{\tr}{\mathrm{tr}}

\newcommand{\p}{\partial}

\newcommand{\UN}{\mathrm{U}(N)}

\newtheorem{theorem}{Theorem}
\newtheorem{lemma}{Lemma}
\newtheorem{corollary}{Corollary}
\newtheorem{remark}{Remark}
\newtheorem{proposition}{Proposition}

\numberwithin{equation}{section}
\begin{document}
\title{Universality in Complex Wishart ensembles: The 2 cut case}
\author{M. Y. Mo}
\date{}
\maketitle
\begin{abstract} We considered $N\times N$ Wishart ensembles
in the class $W_\mathbb{C}\left(\Sigma_N,M\right)$ (complex Wishart
matrices with $M$ degrees of freedom and covariance matrix
$\Sigma_N$) such that $N_0$ eigenvalues of $\Sigma_N$ is 1 and
$N_1=N-N_0$ of them are $a$. We studied the limit as $M$, $N$, $N_0$
and $N_1$ all go to infinity such that $\frac{N}{M}\rightarrow c$,
$\frac{N_1}{N}\rightarrow\beta$ and $0<c,\beta<1$. In this case, the
limiting eigenvalue density can either be supported on 1 or 2
disjoint intervals in $\mathbb{R}_+$. We obtained a necessary and
sufficient condition on the parameters $a$, $\beta$ and $c$ such
that the limiting distribution is supported on 2 disjoint intervals
and have computed the eigenvalue density in the limit. Furthermore,
by using Riemann-Hilbert analysis similar to the one in
\cite{BKext1} (See also \cite{Lysov}), we have shown that under
proper rescaling of the eigenvalues, the limiting correlation kernel
is given by the sine kernel and the Airy kernel in the bulk and the
edge of the spectrum respectively. As a consequence, the behavior of
the largest eigenvalue in this model is described by the Tracy-Widom
distribution.
\end{abstract}
\section{Introduction}
Let $X$ be an $M\times N$ (assuming $M\geq N$) matrix with i.i.d.
complex Gaussian entries whose real and imaginary parts have
variance $\frac{1}{2}$ and zero mean. Let $\Sigma_N$ be an $N\times
N$ positive definite Hermitian matrix with eigenvalues
$a_1,\ldots,a_N$ (not necessarily distinct). In this paper, we will
consider the case where $\Sigma_N$ has only 2 distinct eigenvalues,
1 and $a$ such that $N_1$ of its eigenvalues are $a$ and $N-N_1$ of
them are $1$. We will assume that $\frac{N}{M}\rightarrow c$ and
$\frac{N_1}{N}\rightarrow\beta$ as $N,M\rightarrow\infty$ and that
$0<c,\beta<1$. To be precise, we will assume the following
\begin{equation}\label{eq:lim}
cM-N=\tau_1=O(1),\quad N\beta-N_1=\tau_2=O(1),\quad
M,N,N_1\rightarrow\infty.
\end{equation}
Let $\Sigma_N^{\frac{1}{2}}$ be any Hermitian square root of
$\Sigma_N$. Then the columns of the matrix $X\Sigma_N^{\frac{1}{2}}$
are random vectors with variances $\frac{1}{2}\sqrt{a_j}$. Let the
matrix $B_N$ be the following
\begin{equation}\label{eq:BN}
B_N=\frac{1}{M}\Sigma_N^{\frac{1}{2}}X^{\dag}X\Sigma_N^{\frac{1}{2}},
\end{equation}
Then $B_N$ is the sample covariance matrix of the columns of
$X\Sigma_N^{\frac{1}{2}}$, while $\Sigma_N$ is the covariance
matrix. In particular, $B_N$ is a complex Wishart matrix in the
class $W_{\mathbb{C}}\left(\Sigma_N,M\right)$.

The sample covariance matrix is a fundamental tool in the studies of
multivariate statistics and its distribution is already known to
Wishart at around 1928 (See e.g. \cite{Muir})
\begin{equation}\label{eq:wishart}
\mathcal{P}(B_N)=\frac{1}{C}e^{-M\tr(\Sigma^{-1}B_N)}\left(\det
S\right)^{M-N},\quad M\geq N,
\end{equation}
for some normalization constant $C$.

Let $y_1>\ldots>y_N>0$ be the eigenvalues of the sample covariance
matrix $B_N$ and let $A$, $\Lambda$ be diagonal matrices such that
\begin{equation*}
\begin{split}
&A=diag(a_1,\ldots,a_N),\quad\Lambda=diag(y_1,\ldots,y_N),\\
&B_N=U_B\Lambda U_B^{-1},\quad\Sigma_N=U_AAU_A^{-1},\quad U_A,
U_B\in U(N).
\end{split}
\end{equation*}
Then the eigenvalue density of $B_N$ is given by
\begin{equation}\label{eq:jpdf}
\mathcal{P}(y)=\frac{1}{C}\prod_{i<j}(y_i-y_j)^2\prod_{j=1}^Ny_{j}^{M-N}\int_{Q\in\UN}e^{-M\tr(A^{-1}Q\Lambda
Q^{-1})}dQ,
\end{equation}
where $Q=U_A^{-1}U_B\in U(N)$ is unitary and $dQ$ is the normalized
Haar measure on $U(N)$. The eigenvalue density of $B_N$ can be
simplified by using the Harish-Chandra \cite{HC} (or Itzykson-Zuber
\cite{IZ}) formula
\begin{equation}\label{eq:HCIZ}
\int_{Q\in U(N)}e^{-M\tr(A^{-1}Q\Lambda
Q^{-1})}dQ=\prod_{j=1}^N\Gamma(j)\frac{\det\left[e^{-Ma_i^{-1}y_j}\right]_{1\leq
i,j\leq N}} {\det\left[\left(-Ma_i\right)^{-j+1}\right]_{1\leq
i,j\leq N}\det\left[y_i^{j-1}\right]_{1\leq i,j\leq N}},
\end{equation}
where in the case of $A$ or $\Lambda$ having multiple eigenvalues,
the above should be interpreted by using the L'Hopital rule.

By using (\ref{eq:HCIZ}) in (\ref{eq:jpdf}), we see that in the case
where $N_1$ eigenvalues of $\Sigma_N$ is $a$ and $N-N_1$ is $1$, the
joint probability density function (j.p.d.f) for the eigenvalues of
$B_N$ is given by
\begin{equation}\label{eq:jpdf1}
\mathcal{P}(y)=\frac{1}{Z_{M,N}}\prod_{i<j}(y_i-y_j)\prod_{j=1}^Ny_{j}^{M-N}\det\left[
y_k^{d^N_j-1}e^{-Ma_j^{-1}y_k}\right]_{1\leq j,k,\leq N},
\end{equation}
where $Z_{M,N}$ is a normalization constant and $d^N_j$, $a_j$ are
given by
\begin{equation*}
\begin{split}
d^N_j&=j,\quad a_j=1,\quad 1\leq j\leq N-N_1,\\
d^N_j&=j-N+N_1,\quad a_j=a,\quad N-N_1< j\leq N.
\end{split}
\end{equation*}
In this paper we will study the asymptotic limit of the Wishart
distribution with $\frac{N}{M}\rightarrow c$ and
$\frac{N_1}{N}\rightarrow\beta$ as $M$, $N\rightarrow\infty$ in such
a way that $0<\beta,c<1$. In this case, the empirical distribution
function (e.d.f) $F_N$ of the eigenvalues will converge weakly to a
nonrandom p.d.f. $F$, which will be supported on either 1 or 2
intervals in $\mathbb{R}_+$. By applying the results of \cite{CS} to
our case, we can extract properties of the measure $F$ from the
solutions of an algebraic equation (See Section \ref{se:Stie} for
details)
\begin{equation}\label{eq:curve20}
\begin{split}
za\xi^3&+(A_2z+B_2)\xi^2+(z+B_1)\xi+1=0,\\
A_2&=(1+a),\quad B_2=a(1-c),\\
B_1&=1-c(1-\beta)+a(1-c\beta).
\end{split}
\end{equation}
Our first main result involves a necessary and sufficient condition
for $F$ to be supported on 2 intervals and an explicit formula for
the distribution $F$ (See Theorem \ref{thm:cuts}).
\begin{theorem}\label{thm:main1}
Let $\Delta$ be the discriminant of the quartic polynomial
\begin{equation}\label{eq:quartic1}
\begin{split}
a^2(1-c)\xi^4
&+2(a^2(1-c\beta)+a(1-c(1-\beta))\xi^3\\
&+(1-c(1-\beta)+a^2(1-c\beta)+4a)\xi^2 +2(1+a)\xi+1=0.
\end{split}
\end{equation}
then the support of $F$ consists of 2 disjoint intervals if and only
if $\Delta>0$.
\end{theorem}
When $\Delta>0$, we also have the following expression for the
density function of $F$ (See Theorem \ref{thm:density}).
\begin{theorem}\label{thm:side1}
Let $\Delta$ be the discriminant of the quartic polynomial
(\ref{eq:quartic1}). Suppose $\Delta>0$ and let
$\gamma_1<\ldots<\gamma_4$ be the 4 real roots to
(\ref{eq:quartic1}). Let $\lambda_1<\ldots<\lambda_4$ be the
following
\begin{equation*}
\lambda_k=-\frac{1}{\gamma_k}+c\frac{1-\beta}{1+\gamma_k}+c\frac{a\beta}{1+a\gamma_k},\quad
k=1,\ldots,4.
\end{equation*}
Then the p.d.f $F$ is supported on
$[\lambda_1,\lambda_2]\cup[\lambda_3,\lambda_4]$ with the following
density $d F(z)=\rho(z)dz$
\begin{equation}\label{eq:rho}
\begin{split}
\rho(z)=\frac{3}{2\pi}\left|\left(\frac{r(z)+\sqrt{-\frac{1}{27a^4z^4}D_3(z)}}{2}\right)^{\frac{1}{3}}-\left(\frac{r(z)-\sqrt{-\frac{1}{27a^4z^4}D_3(z)}}{2}\right)^{\frac{1}{3}}\right|,
\end{split}
\end{equation}
where $D_3(z)$ and $r(z)$ are given by
\begin{equation*}
\begin{split}
D_3(z)&=(1-a)^2\prod_{j=1}^4(z-\lambda_j),\\
r(z)&=\frac{1}{27}\Bigg(-\frac{2B_2^3}{a^3}z^{-3}+\left(\frac{9B_1B_2}{a^2}-\frac{6A_2B_2^2}{a^3}\right)z^{-2}+\left(\frac{9B_2}{a^2}+\frac{9B_1A_2}{a^2}-\frac{27}{a}-\frac{6A_2^2B_2}{a^3}\right)z^{-1}\\
&+\left(\frac{9A_2}{a^2}-\frac{2A_2^3}{a^3}\right)\Bigg).
\end{split}
\end{equation*}
The constants $A_1$, $B_1$ and $B_2$ in the above equation are
defined by
\begin{equation*}
A_2=(1+a),\quad B_2=a(1-c),\quad B_1=1-c(1-\beta)+a(1-c\beta).
\end{equation*}
The cube root in (\ref{eq:density}) is chosen such that
$\sqrt[3]{A}\in\mathbb{R}$ for $A\in\mathbb{R}$ and the square root
is chosen such that $\sqrt{A}>0$ for $A>0$.
\end{theorem}
\begin{remark}\label{re:sqrt}
Since $D_3(z)$ can be written as
\begin{equation}
\frac{D_3(z)}{a^4z^4}=-27\left(r(z)\right)^2-4\left(p(z)\right)^3,
\end{equation}
for some polynomial $p(z)$ in $z^{-1}$. We see that if $r(z)$
vanishes at any of the $\lambda_k$, then $D_3(z)$ will have at least
a double root at these points, hence $r(\lambda_k)\neq0$. From this
and (\ref{eq:rho}), we see that the density $\rho(z)$ vanishes like
a square root at the points $\lambda_k$, $k=1,\ldots,4$.
\begin{equation}\label{eq:rhok}
\rho(z)=\frac{\rho_k}{\pi}|z-\lambda_k|^{\frac{1}{2}}+O\left((z-\lambda_k)\right),\quad
z\rightarrow\lambda_k.
\end{equation}
\end{remark}
Having obtained the global statistics of eigenvalues, we will
continue to answer questions about the local statistics for the
eigenvalues. In particular, we will show that, under suitable
scaling of the eigenvalues, the behavior of the largest eigenvalue
is given by the Tracy-Widom distribution.

A result by Baik, Ben-Arous and P\'ech\'e \cite{Baik95} shows that
the correlation functions of the eigenvalues can be expressed in
terms of a Fredholm determinant with kernel $K_{M,N}(x,y)$. In
\cite{BK1} and \cite{DF}, the authors have expressed this kernel in
terms of multiple orthogonal polynomials (See Section \ref{se:MOP}
for details) and have shown that the $m$-point correlation function
for the Wishart distribution (\ref{eq:wishart}) is given by
\begin{equation}\label{eq:mpoint}
\mathcal{R}_{m}^{(M,N)}(y_1,\ldots,y_m)=\det\left(K_{M,N}(y_j,y_k)\right)_{1\leq
j,k\leq m}
\end{equation}
where $\mathcal{R}_{m}^{(M,N)}(y_1,\ldots,y_m)$ is the $m$-point
correlation function
\begin{equation}\label{eq:corre}
\mathcal{R}_m^{(M,N)}(y_1,\ldots,y_m)=\frac{N!}{(N-m)!}\int_{\mathbb{R}_+}\cdots\int_{\mathbb{R}_+}
\mathcal{P}(y)dy_{m+1}\ldots dy_N.
\end{equation}
Our next main result shows the universality of the correlation
kernel when $M$, $N$ and $N_1\rightarrow\infty$. As a polynomial in
$\xi$, the algebraic equation (\ref{eq:curve20}) admits a unique
solution $\xi_1(z)$ that is analytic in the upper half plane and
vanishes at $z=\infty$. If we let $k_2=2$, $k_3=4$ for $a>1$ and
$k_2=4$, $k_3=2$ for $a<1$, then at the points $\lambda_{k_j-1}$ and
$\lambda_{k_j}$, $j=2,3$, the root $\xi_1(z)$ will coincide with
another root $\xi_j(z)$. Let $\theta_j(z)$ be the following,
\begin{equation}
\begin{split}
\theta_1(z)&=\int_{\lambda_4}^z\xi_1(x)dx,\quad
\theta_2(z)=\int_{\lambda_{k_2}}^z\xi_2(x)dx,\quad
\theta_3(z)=\int_{\lambda_{k_3}}^z\xi_3(x)dx,\\
k_2&=2,\quad k_3=4,\quad a>1,\quad k_2=4,\quad k_3=2,\quad a<1.
\end{split}
\end{equation}
where the integration path is taken in the upper half plane. Then we
can define a rescaled kernel $\hat{K}_{M,N}(x,y)$ in a neighborhood
of $[\lambda_{k_j-1},\lambda_{k_j}]$.
\begin{equation}\label{eq:khat}
\hat{K}_{M,N}(x,y)=\left(xy^{-1}\right)^{\frac{M-N}{2}}e^{\frac{M}{2}\left(\theta_{1,+}(x)+
\theta_{j,+}(x)-\theta_{1,+}(y)-\theta_{j,+}(y)\right)}K_{M,N}(x,y),
\end{equation}
where $\theta_{1,+}(x)$ and $\theta_{j,+}(x)$ are the boundary
values of $\theta_1(x)$ and $\theta_j(x)$ on the positive side of
the real axis.

The rescaling from $K_{M,N}(x,y)$ to $\hat{K}_{M,N}(x,y)$ in
(\ref{eq:khat}) does not affect the determinantal formula
(\ref{eq:mpoint}). That is, we have
\begin{equation}\label{eq:mhat}
\mathcal{R}_{m}^{(M,N)}(y_1,\ldots,y_m)=\det\left(\hat{K}_{M,N}(y_j,y_k)\right)_{1\leq
j,k\leq m}
\end{equation}
We then have the following universality result for the kernel
$\hat{K}_{M,N}(x,y)$.
\begin{theorem}\label{thm:main2}
Suppose $\Delta$ in Theorem \ref{thm:main1} is positive. Let
$\rho(z)$ be the density function of $F$ in Theorem \ref{thm:side1}.
Then for any $x_0\in(\lambda_1,\lambda_2)\cup(\lambda_3,\lambda_4)$
and $u$, $v\in\mathbb{R}$, we have
\begin{equation}\label{eq:bulk}
\lim_{N,M\rightarrow\infty}\frac{1}{M\rho(x_0)}\hat{K}_{M,N}\left(x_0+\frac{u}{M\rho(x_0)},
x_0+\frac{v}{M\rho(x_0)}\right)=\frac{\sin \pi(u-v)}{\pi(u-v)}.
\end{equation}
On the other hand, let $x_0=\lambda_k$, $k=1,\ldots,4$, then for
$u$, $v\in\mathbb{R}$, we have
\begin{equation}\label{eq:edge}
\begin{split}
&\lim_{N,M\rightarrow\infty}\frac{1}{\left(M\rho_k\right)^{\frac{2}{3}}}
\hat{K}_{M,N}\left(x_0+\frac{u}{\left(M\rho_k\right)^{\frac{2}{3}}},
x_0+\frac{v}{\left(M\rho_k\right)^{\frac{2}{3}}}\right)=\\
&\frac{\mathrm{Ai}(u)\mathrm{Ai}^{\prime}(v)
-\mathrm{Ai}^{\prime}(u)\mathrm{Ai}(v)}{u-v},
\end{split}
\end{equation}
where $\mathrm{Ai}(z)$ is the Airy function and $\rho_k$,
$k=1,\ldots,4$ are the constants in (\ref{eq:rhok}).
\end{theorem}
Recall that the Airy function is the unique solution to the
differential equation $v^{\prime\prime}=zv$ that has the following
asymptotic behavior as $z\rightarrow\infty$ in the sector
$-\pi+\epsilon\leq \arg(z)\leq \pi-\epsilon$, for any $\epsilon>0$.
\begin{equation}\label{eq:asymairy}
\mathrm{Ai}(z)=\frac{1}{2\sqrt{\pi}z^{\frac{1}{4}}}e^{-\frac{2}{3}z^{\frac{3}{2}}}\left(1+O(z^{-\frac{3}{2}})\right)
,\quad -\pi+\epsilon\leq \arg(z)\leq \pi-\epsilon,\quad
z\rightarrow\infty.
\end{equation}
where the branch cut of $z^\frac{3}{2}$ in the above is chosen to be
the negative real axis.

Since the limiting kernel $\hat{K}_{M,N}(x,y)$ takes the form of the
Airy kernel (\ref{eq:edge}), by a well-known result of Tracy and
Widom \cite{TW1}, we have the following
\begin{theorem}\label{thm:TW}
Let $y_1$ be the largest eigenvalue of $B_N$, then we have
\begin{equation}
\lim_{M,N\rightarrow\infty}
\mathbb{P}\left(\left(y_1-\lambda_4\right)\left(M\rho_4\right)^{\frac{2}{3}}<s\right)
=TW(s),
\end{equation}
where $TW(s)$ is the Tracy-Widom distribution
\begin{equation*}
TW(s)=\exp\left(-\int_{s}^{\infty}(t-s)q^2(t)dt\right),
\end{equation*}
and $q(s)$ is the solution of Painlev\'e II equation
\begin{equation*}
q^{\prime\prime}(s)=sq(s)+2q^3(s),
\end{equation*}
with the following asymptotic behavior as $s\rightarrow\infty$.
\begin{equation*}
q(s)\sim-\mathrm{Ai}(s),\quad s\rightarrow +\infty.
\end{equation*}
\end{theorem}
\begin{remark} The results obtained in this paper are derived from the asymptotics of
multiple Laguerre polynomials (See Section \ref{se:MOP} for
details). In \cite{Lysov}, strong asymptotics of the multiple
Laguerre polynomials was also obtained for the case when
$1-c=O(M^{-1})$ and $\beta=\frac{1}{2}$. This corresponds to the
case when $M-N$ is finite. In \cite{Lysov}, it was shown that, in
that case, the origin will become an edge of the spectrum and the
asymptotics of the multiple Laguerre polynomials are described by
the Bessel functions near the origin. The statistical implication of
\cite{Lysov} is that, when $M-N$ is finite, the correlation kernel
near the origin will be given by the Bessel kernel (See, e.g.
\cite{TW2}) instead of the Airy kernel (\ref{eq:edge}). This is an
analogue to the situation when $\Sigma_N=I$ and $M-N$ is finite,
where the correlation kernel at the origin is also described by the
Bessel kernel (See \cite{DF} and \cite{V}).
\end{remark}

Until recently, most of the universality results for the Wishart
distribution was obtained when the covariance matrix $\Sigma_N$ is
the identity matrix \cite{El03}, \cite{Fo93}, \cite{J} and
\cite{Jo}. More recent studies have extended these results to the
spiked model proposed by Johnstone \cite{J}, in which $\Sigma_N$ is
a finite perturbation of the identity matrix \cite{BaikD},
\cite{Baik95}, \cite{Baikspike}, \cite{DF}, \cite{W1}, \cite{W2}.
This is the first few cases when universality results was obtained
for a covariance matrix that is not a finite perturbation of the
identity matrix. (See also \cite{El} in which a different class of
$\Sigma_N$ was studied) For theoretical reasons, the model studied
in this paper is crucial in understanding of the phase transition
behavior that occurs in Wishart ensembles. (See \cite{Baik95}). When
the 2 intervals in the support of $F$ closes up, a phase transition
takes place and the correlation kernel will demonstrate a different
behavior at the point where the support closes up. With the
Riemann-Hilbert technique used in this paper, such behavior can be
studied rigorously as in \cite{BKdou} (See also \cite{Lysov}). We
plan to study this phenomenon in a further publication. For
practical reasons, many covariance matrices appearing in fields of
science are not finite perturbations of the identity matrix. In
fact, covariance matrices that have groups of distinct eigenvalues
are accepted models in various areas such as the correlation of
genes in microarray analysis and the correlation of the returns of
stocks in finance.

\subsection*{Acknowledgement}
The author acknowledges A. Kuijlaars for pointing out reference
\cite{Lysov} to me and EPSRC for the financial support provided by
the grant EP/D505534/1.

\section{Multiple Laguerre polynomials and the correlation
kernel}\label{se:MOP}

The main tool in our analysis involves the use of multiple
orthogonal polynomials and the Riemann-Hilbert problem associated
with them. In this section we shall recall the results in \cite{BK1}
and \cite{DF} and express the correlation kernel $K_{M,N}(x,y)$ in
(\ref{eq:mpoint}) in terms of the multiple Laguerre polynomials. In
Section \ref{se:RHP}, we will apply Riemann-Hilbert analysis to
obtain the asymptotics of these multiple Laguerre polynomials and
use them to prove Theorem \ref{thm:main2}.

We shall not define the multiple Laguerre polynomials in the most
general setting, but only define the ones that are relevant to our
set up. Readers who are interested in the theory of multiple
orthogonal polynomials can consult the papers \cite{Ap1},
\cite{Ap2}, \cite{BK1}, \cite{vanGerKuij}. Let $L_{n_1,n_2}(x)$ be
the monic polynomial such that
\begin{equation}\label{eq:multiop}
\begin{split}
&\int_{0}^{\infty}L_{n_1,n_2}(x)x^{i+M-N}e^{-Mx}dx=0,\quad i=0,\ldots,n_1-1,\\
&\int_0^{\infty}L_{n_1,n_2}(x)x^{i+M-N}e^{-Ma^{-1}x}dx=0,\quad
i=0,\ldots, n_2-1.
\end{split}
\end{equation}
and let $Q_{n_1,n_2}(x)$ be a function of the form
\begin{equation}\label{eq:qform}
Q_{n_1,n_2}(x)=A^1_{n_1,n_2}(x)e^{-Mx}+A^a_{n_1,n_2}(x)e^{-Ma^{-1}x},
\end{equation}
where $A^1_{n_1,n_2}(x)$ and $A^a_{n_1,n_2}(x)$ are polynomials of
degrees $n_1-1$ and $n_2-1$ respectively, and that $Q_{n_1,n_2}(x)$
satisfies the following
\begin{equation}\label{eq:qorth}
\int_0^{\infty}x^iQ_{n_1,n_2}(x)x^{M-N}dx=\left\{
                                            \begin{array}{ll}
                                              0, & \hbox{$i=0,\ldots,n_1+n_2-2$;} \\
                                              1, & \hbox{$i=n_1+n_2-1$.}
                                            \end{array}
                                          \right.
\end{equation}
The polynomial $L_{n_1,n_2}(x)$ is called the multiple Laguerre
polynomial of type II and the polynomials $A^1_{n_1,n_2}(x)$ and
$A^a_{n_1,n_2}(x)$ are called multiple Laguerre polynomials of type
I (with respect to the weights $x^{M-N}e^{-Mx}$ and
$x^{M-N}e^{-Ma^{-1}x}$) \cite{Ap1}, \cite{Ap2}. These polynomials
exist and are unique. Moreover, they admit integral expressions
\cite{BK1}.

Let us define the constants $h^{(1)}_{n_1,n_2}$ and
$h^{(2)}_{n_1,n_2}$ to be
\begin{equation}\label{eq:normalize}
\begin{split}
h^{(1)}_{n_1,n_2}&=\int_{0}^{\infty}L_{n_1,n_2}(x)x^{n_1+M-N}e^{-Mx}dx,\\
h^{(2)}_{n_1,n_2}&=\int_{0}^{\infty}L_{n_1,n_2}(x)x^{n_2+M-N}e^{-Ma^{-1}x}dx.
\end{split}
\end{equation}

Then the following result in \cite{BK1} and \cite{DF} allows us to
express the correlation kernel in (\ref{eq:mpoint}) in terms of a
finite sum of the multiple Laguerre polynomials.
\begin{proposition}\label{pro:CD}
The correlation kernel in $K_{M,N}(x,y)$ (\ref{eq:mpoint}) can be
expressed in terms of multiple Laguerre polynomials as follows
\begin{equation}\label{eq:ker}
\begin{split}
\left(xy\right)^{\frac{N-M}{2}}(x-y)K_{M,N}(x,y)&=L_{N_0,N_1}(x)Q_{N_0,N_1}(x)\\
&-\frac{h^{(1)}_{N_0,N_1}}{h^{(1)}_{N_0-1,N_1}}L_{N_0-1,N_1}(x)Q_{N_0+1,N_1}(x)\\
&-\frac{h^{(2)}_{N_0,N_1}}{h^{(2)}_{N_0,N_1-1}}L_{N_0,N_1-1}(x)Q_{N_0,N_1+1}(x)
\end{split}
\end{equation}
where $N_0=N-N_1$.
\end{proposition}
This result allows us to compute the limiting kernel once we obtain
the asymptotics for the multiple Laguerre polynomials. In this
paper, we will use the Riemann-Hilbert method to obtain such
asymptotics and use them to compute the limiting kernel and to prove
Theorem 2. The Riemann-Hilbert analysis used in this paper involves
a $3\times 3$ Riemann-Hilbert problem and the analysis is similar to
that in \cite{BKext1} where random matrices with external source was
studied (See also \cite{Lysov}). In what follows, we will restrict
ourselves to the case when the limiting eigenvalue distribution is
supported on 2 disjoint intervals. The case when the limiting
distribution is supported on 1 interval will be considered in a
separate publication.

\section{Stieltjes transform of the eigenvalue
distribution}\label{se:Stie}

In order to study the asymptotics of the correlation kernel, we
would need to know the asymptotic eigenvalue distribution of the
Wishart ensemble (\ref{eq:wishart}). Let $F_N(x)$ be the empirical
distribution function (e.d.f) of the eigenvalues of $B_N$
(\ref{eq:BN}). The asymptotic properties of $F_N(x)$ can be found by
studying its Stieltjes transform.

The Stieltjes transform of a probability distribution function
(p.d.f) $G(x)$ is defined by
\begin{equation}\label{eq:stie}
m_G(z)=\int_{-\infty}^{\infty}\frac{1}{\lambda-z}dG(x), \quad
z\in\mathbb{C}^+=\left\{z\in\mathbb{C}:\mathrm{Im}(z)>0\right\}.
\end{equation}
Given the Stieljes transform, the p.d.f can be found by the
inversion formula
\begin{equation}\label{eq:inver}
G([a,b])=\frac{1}{\pi}\lim_{\epsilon\rightarrow 0^+}\int_{a}^bIm
\left(m_G(s+i\epsilon)\right)ds.
\end{equation}
The properties of the Stieltjes transform of $F_N(x)$ has been
studied by Bai, Silverstein and Choi in the papers \cite{CS},
\cite{BS95}, \cite{BS98}, \cite{BS99}, \cite{S95}. We will now
summarize the results that we need from these papers.

First let us denote the e.d.f of the eigenvalues of $\Sigma_N$ by
$H_N$, that is, we have
\begin{equation*}
dH_N(x)=\frac{1}{N}\sum_{j=1}^N\delta_{a_j}.
\end{equation*}
Furthermore, we assume that as $N\rightarrow\infty$, the
distribution $H_N$ weakly converges to a distribution function $H$.
Then as $N\rightarrow\infty$, the e.d.f $F_N(x)$ converges weakly to
a nonrandom p.d.f $F$, and that the Stieltjes transform $m_F$ of
$F(x)$ satisfies the following equation \cite{CS}, \cite{S95}
\begin{equation}\label{eq:algeq}
m_F(z)=\int_{\mathbb{R}}\frac{1}{t(1-c-czm_F)-z}dH(t).
\end{equation}
Let us now consider the closely related matrix $\underline{B}_N$
\begin{equation}\label{eq:BNline}
\underline{B}_N=\frac{1}{N}X\Sigma_NX^{\dag}.
\end{equation}
The matrix $\underline{B}_N$ has the same eigenvalues as $B_N$
together with an addition $M-N$ zero eigenvalues. Therefore the
e.d.f $\underline{F}_N$ of the eigenvalues of $\underline{B}_N$ are
related to $F_N$ by
\begin{equation}\label{eq:FNline}
\underline{F}_N=(1-c_N)I_{[0,\infty)}+c_NF_N,\quad c_N=\frac{N}{M}.
\end{equation}
where $I_{[0,\infty)}$ is the step function that is $0$ on
$\mathbb{R}_-$ and $1$ on $\mathbb{R}_+$. In particular, as
$N\rightarrow\infty$, the distribution $\underline{F}_N$ converges
weakly to a p.d.f $\underline{F}$ that is related to $F$ by
\begin{equation}\label{eq:Fline}
\underline{F}=(1-c)I_{[0,\infty)}+cF
\end{equation}
and their Stieltjes transforms are related by
\begin{equation}\label{eq:stielim}
m_{\underline{F}}(z)=-\frac{1-c}{z}+cm_F(z).
\end{equation}
Then from (\ref{eq:algeq}), we see that the Stieltjes transform
$m_{\underline{F}}(z)$ satisfies the following equation
\begin{equation}\label{eq:meq}
m_{\underline{F}}(z)=-\left(z-c\int_{\mathbb{R}}\frac{tdH(t)}{1+tm_{\underline{F}}}\right)^{-1}.
\end{equation}
This equation has an inverse \cite{BS98}, \cite{BS99}
\begin{equation}\label{eq:zeq}
z(m_{\underline{F}})=-\frac{1}{m_{\underline{F}}}+c\int_{\mathbb{R}}\frac{t}{1+tm_{\underline{F}}}dH(t).
\end{equation}

\subsection{Riemann surface and the Stieltjes transform}
We will now restrict ourselves to the case when the matrix
$\Sigma_N$ has 2 distinct eigenvalues only. Without lost of
generality, we will assume that one of these values is 1 and the
other one is $a>0$. Let $0<\beta<1$, we will assume that as
$N\rightarrow\infty$, $N_1$ of the eigenvalues take the value $a$
and $N_0=N-N_1$ of the eigenvalues are 1 and that
$\frac{N_1}{N}\rightarrow\beta$. That is, as $N\rightarrow\infty$,
the e.d.f $H_N(x)$ converges to the following
\begin{equation}\label{eq:limedf}
dH_N(x)\rightarrow dH(x)=(1-\beta)\delta_1+\beta\delta_{a}.
\end{equation}
By substituting this back into the (\ref{eq:zeq}), we see that the
Stieltjes transform $\xi(z)=m_{\underline{F}}(z)$ is a solution to
the following algebraic equation
\begin{equation}\label{eq:curve1}
z=-\frac{1}{\xi}+c\frac{1-\beta}{1+\xi}+c\frac{a\beta}{1+a\xi}.
\end{equation}
Rearranging the terms, we see that $\xi=m_{\underline{F}}$ solves
the following
\begin{equation}\label{eq:curve2}
\begin{split}
za\xi^3&+(A_2z+B_2)\xi^2+(z+B_1)\xi+1=0,\\
A_2&=(1+a),\quad B_2=a(1-c),\\
B_1&=1-c(1-\beta)+a(1-c\beta).
\end{split}
\end{equation}
This defines a Riemann surface $\Lie$ as a 3-folded cover of the
complex plane.

By solving the cubic equation (\ref{eq:curve2}) or by analyzing the
asymptotic behavior of the equation as $z\rightarrow\infty$, we see
that these solutions have the following behavior as
$z\rightarrow\infty$.
\begin{equation}\label{eq:xiinfty}
\begin{split}
\xi_1(z)&=-\frac{1}{z}+O(z^{-2}), \quad z\rightarrow\infty,\\
\xi_2(z)&=-1+\frac{c(1-\beta)}{z}+O(z^{-2}), \quad z\rightarrow\infty,\\
\xi_3(z)&=-\frac{1}{a}+\frac{c\beta}{z}+O(z^{-2}), \quad
z\rightarrow\infty.
\end{split}
\end{equation}
On the other hand, as $z\rightarrow 0$, the 3 branches of $\xi(z)$
behave as follows
\begin{equation}\label{eq:zeroasym}
\begin{split}
\xi_{\alpha}(z)&=-\frac{1-c}{z}+O(1),\quad z\rightarrow 0,\\
\xi_{\beta}(z)&=R_1+O(z),\quad z\rightarrow 0,\\
\xi_{\gamma}(z)&=R_2+O(z),\quad z\rightarrow 0.
\end{split}
\end{equation}
where the order of the indices $\alpha$, $\beta$ and $\gamma$ does
not necessarily coincides with the ones in (\ref{eq:xiinfty}) (i.e.
we do not necessarily have $\alpha=1$, $\beta=2$ and $\gamma=3$).
The constants $R_1$ and $R_2$ are the two roots of the quadratic
equation
\begin{equation}\label{eq:quad}
a(1-c)x^2+(1-c(1-\beta)+a(1-c\beta))x+1=0.
\end{equation}
The discriminant $D_2$ of (\ref{eq:quad}) is given by
\begin{equation}\label{eq:D2}
D_2=\left(1-c(1-\beta)+a(1-c\beta)\right)^2-4a(1-c),
\end{equation}
for $a<1$, it is an strictly increasing function in $\beta$ and
hence
\begin{equation*}
D_2>\left(1-c+a\right)^2-4a(1-c)=(1-c-a)^2\geq 0,
\end{equation*}
and for $a>1$, it is a strictly decreasing function in $\beta$ and
hence
\begin{equation*}
D_2>\left(1+a(1-c)\right)^2-4a(1-c)=(1-a(1-c))^2\geq 0.
\end{equation*}
Therefore both $R_1$ and $R_2$ are real as they should be.

The functions $\xi_j(z)$ will not be analytic at the branch points
of $\Lie$ and they will be discontinuous across the branch cuts
joining these branch points. Moreover, from (\ref{eq:zeroasym}), one
of the functions $\xi_j(z)$ will have a simple pole at $z=0$. Apart
from these singularities, however, the functions $\xi_j(z)$ are
analytic.

\subsection{Sheet structure of the Riemann surface}
In this section we will study the sheet structure of the Riemann
surface $\Lie$. As we shall see, the branch $\xi_1(z)$ turns out to
be the Stieljes transform $m_{\underline{F}}(z)$ and its branch cut
will become the support of $\underline{F}$.

\subsubsection{The support of eigenvalues}

The branch points of the Riemann surface $\Lie$ are the points on
$\Lie$ in which $\frac{d z}{d\xi}$ vanishes. They are also potential
end points of $\mathrm{supp}(\underline{F})$ due to the following
result in \cite{CS}.
\begin{lemma}\label{le:CS}$\cite{CS}$ If $z\notin\mathrm{supp}(\underline{F})$,
then $m=m_{\underline{F}}(z)$ satisfies the following.
\begin{enumerate}
\item $m\in\mathbb{R}\setminus\{0\}$;
\item $-\frac{1}{m}\notin\mathrm{supp}(H)$;
\item Let $z$ be defined by (\ref{eq:zeq}), then $z^{\prime}(m)>0$, where the prime denotes
the derivative with respect to $m_{\underline{F}}$ in
(\ref{eq:zeq}).
\end{enumerate}
Conversely, if $m$ satisfies 1-3, then
$z=z(m)\notin\mathrm{supp}(\underline{F})$.
\end{lemma}
This lemma allows us to identify the complement of
$\mathrm{supp}(\underline{F})$ by studying the real points $m$ such
that $z^{\prime}(m)>0$.

Let us differentiate (\ref{eq:curve1}) to obtain an expression of
$\frac{d z}{d\xi}$ in terms of $\xi$.
\begin{equation}\label{eq:zprime}
\begin{split}
\frac{d
z}{d\xi}&=\frac{1}{\xi^2(1+\xi)^2(1+a\xi)^2}\Bigg(a^2(1-c)\xi^4
+2(a^2(1-c\beta)+a(1-c(1-\beta))\xi^3\\
&+(1-c(1-\beta)+a^2(1-c\beta)+4a)\xi^2 +2(1+a)\xi+1\Bigg).
\end{split}
\end{equation}
In particular, the values of $\xi$ at $\frac{d z}{d\xi}=0$
correspond to the roots of the quartic equation
\begin{equation}\label{eq:quart}
\begin{split}
a^2(1-c)\xi^4
&+2(a^2(1-c\beta)+a(1-c(1-\beta))\xi^3\\
&+(1-c(1-\beta)+a^2(1-c\beta)+4a)\xi^2 +2(1+a)\xi+1=0
\end{split}
\end{equation}
Let $\Delta$ be the discriminant of this quartic polynomial, then
when $\Delta>0$, the equation (\ref{eq:quart}) will have 4 distinct
real roots $\gamma_1<\ldots<\gamma_4$. Since $0<c,\beta<1$, the
coefficients of (\ref{eq:quart}) are all positive and hence all
$\gamma_k$ are negative.

Let $\lambda_k$ be the corresponding points in the $z$-plane
\begin{equation}\label{eq:lambda}
\lambda_k=-\frac{1}{\gamma_k}+c\frac{1-\beta}{1+\gamma_k}+c\frac{a\beta}{1+a\gamma_k},\quad
k=1,\ldots,4.
\end{equation}
Note that, from the behavior of $z(\xi)$ in (\ref{eq:curve1}), we
see that near the points $-1$ and $-\frac{1}{a}$, the function
$z(\xi)$ behaves as
\begin{equation}\label{eq:zsing}
\begin{split}
z(\xi)&=\frac{c(1-\beta)}{1+\xi}+O(1),\quad \xi\rightarrow -1,\\
z(\xi)&=\frac{ca\beta}{1+a\xi}+O(1),\quad \xi\rightarrow
-\frac{1}{a}.
\end{split}
\end{equation}
and hence $z^{\prime}(\xi)$ is negative near these points. From this
and (\ref{eq:zprime}), we see that $z^{\prime}(\xi)>0$ on the
intervals $(-\infty,\gamma_1]$, $[\gamma_2,\gamma_3]$,
$[\gamma_4,0)$ and $(0,\infty)$ and none of the points $-1$ or
$-\frac{1}{a}$ belongs to these intervals. On $[\gamma_1,\gamma_2]$
and $[\gamma_3,\gamma_4]$, $z^{\prime}(\xi)$ is negative.

The images of these intervals under the map (\ref{eq:curve1}) then
give us the complement of $\mathrm{supp}(\underline{F})$ in the
$z$-plane. Let us study these images
\begin{lemma}\label{le:image}
The intervals $(-\infty,\gamma_1]$, $[\gamma_2,\gamma_3]$,
$[\gamma_4,0)$ and $(0,\infty)$ are mapped by $z(\xi)$ to
$(0,\lambda_1]$, $[\lambda_2,\lambda_3]$, $[\lambda_4,\infty)$ and
$(-\infty,0)$ respectively. Furthermore, we have
$\lambda_2<\lambda_3$.
\end{lemma}
\begin{proof} Since none of the points $-1$, $-\frac{1}{a}$ and $0$
belongs to these intervals both $z(\xi)$ and $z^{\prime}(\xi)$ are
continuous on these intervals. Moreover, $z(\xi)$ is strictly
increasing on these intervals. Therefore the images of these
intervals are given by
\begin{equation*}
\begin{split}
z\left((-\infty,\gamma_1]\right)&=(z(-\infty),z(\gamma_1)]=(0,\lambda_1]\\
z\left([\gamma_2,\gamma_3]\right)&=[z(\gamma_2),z(\gamma_3)]=[\lambda_2,\lambda_3]\\
z\left([\gamma_4,0)\right)&=[z(\gamma_4),z(0^-))=[\lambda_4,\infty)\\
z\left((0,\infty)\right)&=(z(0^+),z(\infty))=(-\infty,0).
\end{split}
\end{equation*}
where the $\pm$ superscripts in the above indicates that the
function is evaluated at $\pm\epsilon$ for $\epsilon\rightarrow 0$.
Finally, since $\gamma_3>\gamma_2$, we see that
$\lambda_3>\lambda_2$. This concludes the proof.
\end{proof}
Therefore the complement of $\mathrm{supp}(\underline{F})$ is given
by (recall that $\underline{F}$ has a point mass at $0$)
\begin{equation}\label{eq:suppc}
\mathrm{supp}(\underline{F})^c=(-\infty,0)\cup(0,\lambda_1)\cup(\lambda_2,\lambda_3)\cup(\lambda_4,\infty).
\end{equation}
Thus if $\lambda_1<\lambda_2$ and $\lambda_3<\lambda_4$, the support
of $\underline{F}$ will consist of the 2 intervals
$[\lambda_1,\lambda_2]$ and $[\lambda_3,\lambda_4]$. We would like
to show that whenever $\Delta>0$, we have
$\lambda_1<\lambda_2<\lambda_3<\lambda_4$. To do this, let us take a
look at the zeros of the function $z^{\prime}(\xi)$ from the point
of view of branch points.

The $\lambda_k$ are the $z$-coordinates of the zeros of
$z^{\prime}(\xi)$ on $\Lie$. We can treat (\ref{eq:curve2}) as a
polynomial in $\xi$ then $\lambda_k$ will be the zeroes of its
discriminant $D_3(z)=(az)^4\prod_{i<j}(\xi_i-\xi_j)^2$.
\begin{equation}\label{eq:D3}
\begin{split}
D_3(z)&=(1-a)^2z^4+(2A_2^2B_1+2A_2B_2-4A_2^3-12aB_1+18aA_2)z^3\\
&+(B_2^2+A_2^2B_1^2+4A_2B_1B_2-12A_2^2B_2-12aB_1^2+18aB_2+18aA_2B_1-27a^2)z^2\\
&+(2B_1B_2^2+2A_2B_2B_1^2-12A_2B_2^2-4B_1^3a+18aB_1B_2)z+B_1^2B_2^2-4B_2^3.
\end{split}
\end{equation}
The zeros of (\ref{eq:D3}) also correspond to the branch points of
the Riemann surface $\Lie$. These branch points are given on $\Lie$
by $(\lambda_k,\gamma_k)$, for $k=1,\ldots, 4$.

Let us rename the $\lambda_k$ as $z_k$.
\begin{equation*}
\left\{\lambda_1,\ldots,\lambda_4\right\}=\left\{z_1,\ldots,z_4\right\},
\end{equation*}
where the above equality is only an equality between the sets. The
ordering of the points does not necessarily coincide. We shall order
the points $z_j$ such that $z_1<z_2<z_3<z_4$.

Let $J$ be the union of the intervals $[z_1,z_2]$ and $[z_3,z_4]$.
Since the leading coefficient of $D_3(z)$ is $(1-a)^2>0$, we see
that the sign of $D_3(z)$ and hence the 3 roots of the cubic
polynomial (\ref{eq:curve2}) behave as follows for $z\in\mathbb{R}$.
\begin{equation}\label{eq:realim}
\begin{split}
&1. \quad z\in\mathbb{R}\setminus J,\quad D_3(z)>0\Rightarrow
\textrm{$\xi$ has 3 distinct real roots} \\
&2. \quad z\in J,\quad D_3(z)<0\Rightarrow\textrm{$\xi$ has 1 real
and 2 complex roots}.
\end{split}
\end{equation}
In particular, since the roots coincide at the branch points, the
$\gamma_j$ are the values of the double roots of the cubic
(\ref{eq:curve2}) when $z=\lambda_j$. We then have the following
lemma.
\begin{lemma}\label{le:singu}
If $a>1$, then $-1\in[\gamma_1,\gamma_2]$ and
$-\frac{1}{a}\in[\gamma_3,\gamma_4]$. On the other hand, if $a<1$,
then $-\frac{1}{a}\in[\gamma_1,\gamma_2]$ and
$-1\in[\gamma_3,\gamma_4]$. This means that $z(\xi)$ has a
singularity on both of the intervals $[\gamma_1,\gamma_2]$ and
$[\gamma_3,\gamma_4]$.
\end{lemma}
\begin{proof} We will show that $z(\xi)$ cannot be continuous on
either $[\gamma_1,\gamma_2]$ or $[\gamma_3,\gamma_4]$. This means
that both intervals must contain a singularity of $z(\xi)$, which is
one of the points $-1$, $-\frac{1}{a}$. This would imply the lemma.

Let us assume that none of the points $-1$ or $-\frac{1}{a}$ belongs
to $[\gamma_1,\gamma_2]$. Then the function $z(\xi)$ is continuous
on $[\gamma_1,\gamma_2]$. Moreover, from the behavior of $z$ at
these points (\ref{eq:zsing}), we see that none of these points
belongs to $[\gamma_2,\gamma_3]$ either and hence $z(\xi)$ is
continuous on $[\gamma_1,\gamma_3]$. We also seen in the proof of
Lemma \ref{le:image} that $\lambda_2<\lambda_3$. Let us also note
that $\lambda_1<\lambda_4$. If not, then (\ref{eq:suppc}) would
imply that the support of $\underline{F}$ is empty, which cannot
happen. Therefore there are only 3 possible ways of ordering the
points $\lambda_1$, $\lambda_2$ and $\lambda_3$.
\begin{equation}\label{eq:cases}
\begin{split}
&1. \quad\lambda_3>\lambda_1>\lambda_2;\\
&2.\quad\lambda_1>\lambda_3>\lambda_2;\\
&3.\quad\lambda_3>\lambda_2>\lambda_1.
\end{split}
\end{equation}
We will show that $z(\xi)$ cannot be continuous on
$[\gamma_1,\gamma_3]$ in any of these cases.

First note that by the remark after (\ref{eq:zsing}), we see that
$z^{\prime}(\xi)$ is negative between $\gamma_1$ and $\gamma_2$ and
hence we must have $\lambda_1>\lambda_2$ if $z(\xi)$ is continuous
on $[\gamma_1,\gamma_2]$. We therefore rule out the third case in
(\ref{eq:cases}).

Let us now assume $\lambda_3>\lambda_1>\lambda_2$. Then by
(\ref{eq:realim}) and the fact that $\lambda_1<\lambda_4$, we see
that for $z_0\in[\lambda_2,\lambda_1]$, there is only one real $m$
such that $z_0=z(m)$. Let $z_0\in[\lambda_2,\lambda_1]$. Since
$z(\xi)$ is continuous between $\gamma_1$ and $\gamma_2$, the
interval $[\lambda_2,\lambda_1]$ is in the image of
$[\gamma_1,\gamma_2]$ under $z(\xi)$. Hence there is at least one
point $m_0$ in $[\gamma_1,\gamma_2]$ such that $z(m_0)=z_0$.
Similarly, since $z(\xi)$ is continuous between $\gamma_2$ and
$\gamma_3$, the interval $[\lambda_2,\lambda_3]$ is in the image of
$[\gamma_2,\gamma_3]$ under $z(\xi)$. Since
$\lambda_3>\lambda_1>\lambda_2$, the interval
$[\lambda_1,\lambda_2]$ is a subset of $[\lambda_2,\lambda_3]$ and
hence there is at least one point $m_1\in[\gamma_2,\gamma_3]$ such
that $z(m_1)=z_0$. This contradicts the fact that there is only one
real $m$ such that $z_0=z(m)$ in $[\lambda_2,\lambda_1]$.

Let us now consider the case when $\lambda_1>\lambda_3>\lambda_2$.
In this case, for $z_0\in[\lambda_2,\lambda_3]$, there is only one
real $m$ such that $z_0=z(m)$. Again, by continuity of $z(\xi)$ on
$[\gamma_2,\gamma_3]$ and $[\gamma_1,\gamma_2]$, we see that there
is at least one point $m_0$ in $[\gamma_2,\gamma_3]$ and another
point $m_1\in[\gamma_1,\gamma_2]$ such that $z(m_1)=z(m_0)=z_0$.
This leads to a contradiction and hence this cannot be the case
either.

By using the same argument, we can show that $z(\xi)$ cannot be
continuous on $[\gamma_2,\gamma_4]$ and hence one of the points
$-1$, $-\frac{1}{a}$ must be in $[\gamma_1,\gamma_2]$ while the
other one is in $[\gamma_3,\gamma_4]$. This implies the lemma.
\end{proof}

We can now show that the points $\lambda_k$ are ordered as
$\lambda_1<\lambda_2<\lambda_3<\lambda_4$.
\begin{lemma}\label{le:order}
The $\lambda_k$ satisfies $\lambda_1<\lambda_2<\lambda_3<\lambda_4$.
\end{lemma}
\begin{proof} As we have seen in the proof of Lemma \ref{le:singu},
there are only 3 possible ways of ordering the points $\lambda_1$,
$\lambda_2$ and $\lambda_3$, which is indicated in (\ref{eq:cases}).
We will show that the cases 1 and 2 in (\ref{eq:cases}) are not
possible.

Let us assume $\lambda_3>\lambda_1>\lambda_2$. By Lemma
\ref{le:singu}, there is a singularity of $z(\xi)$ in
$[\gamma_1,\gamma_2]$. Let us call this singularity $s_0$. Let
$z_0\in[\lambda_2,\lambda_1]$. Then there is only one real
$m\in\mathbb{R}\setminus\{0,-1,-\frac{1}{a}\}$ such that $z(m)=z_0$.
By the continuity of $z(\xi)$ on $[\gamma_2,\gamma_3]$, there exists
a point $m_0\in[\gamma_2,\gamma_3]$ such that $z(m_0)=z_0$. On the
other hand, let $R>\lambda_1$, since $z(\xi)$ is continuous on
$(s,\gamma_2]$ and $z(s+\epsilon)\rightarrow +\infty$ for small
$\epsilon>0$, there exists $\delta>0$ such that $[\lambda_2,R]$ is
contained in the image of $[s_0+\delta,\gamma_2]$ under the map
$z(\xi)$. In particular, there exists another $m_1$ on
$[s_0+\delta,\gamma_2]$ such that $z(m_1)=z_0$. This leads to a
contradiction.

Let us now assume $\lambda_1>\lambda_3>\lambda_2$. Then for
$z_0\in[\lambda_2,\lambda_3]$, there can only be one real
$m\in\mathbb{R}\setminus\{0,-1,-\frac{1}{a}\}$ such that $z(m)=z_0$.
Applying the continuity argument, we again see that there is an
$m_0\in[\gamma_2,\gamma_3]$ such that $z(m_0)=z_0$ and that for any
$R>\lambda_3$, there exists $\delta>0$ such that $[\lambda_2,R]$ is
contained in the image of $[s_0+\delta,\gamma_2]$ and hence there is
another point $m_1\in[s_0+\delta,\gamma_2]$ such that $z(m_1)=z_0$.
This again leads to a contradiction and hence we must have
$\lambda_3>\lambda_2>\lambda_1$.

By carrying out the same argument for the points $\lambda_2$,
$\lambda_3$ and $\lambda_4$, we see that the only possible ordering
of these points is $\lambda_2<\lambda_3<\lambda_4$. Hence we must
have $\lambda_1<\lambda_2<\lambda_3<\lambda_4$.
\end{proof}
We therefore have the following condition for the support to
consists of 2 intervals.
\begin{theorem}\label{thm:cuts}
Let $\Delta$ be the discriminant of the quartic polynomial
\begin{equation}\label{eq:quartic}
\begin{split}
a^2(1-c)\xi^4
&+2(a^2(1-c\beta)+a(1-c(1-\beta))\xi^3\\
&+(1-c(1-\beta)+a^2(1-c\beta)+4a)\xi^2 +2(1+a)\xi+1=0.
\end{split}
\end{equation}
then the support of $\underline{F}$ consists of 2 disjoint intervals
if and only if $\Delta>0$.
\end{theorem}
\begin{proof} By Lemma \ref{le:order} and (\ref{eq:suppc}),
we see that if $\Delta>0$, then the support of $\underline{F}$
consists of 2 disjoint intervals. If $\Delta\leq 0$, then there can
be at most 3 distinct roots to the equation (\ref{eq:zprime}) and
hence there can be at most 3 end points to the support of
$\underline{F}$. This means that $\underline{F}$ cannot be supported
on 2 disjoint intervals.
\end{proof}
From this theorem, we see that all the $\lambda_k$ are positive.
\begin{corollary}\label{cor:posit}
The $\lambda_k$ satisfies $0<\lambda_1<\ldots<\lambda_4$.
\end{corollary}
\begin{proof}From Theorem \ref{thm:cuts}, we see that the limiting
eigenvalue distribution is supported on
$[\lambda_1,\lambda_2]\cup[\lambda_3,\lambda_4]$. Therefore all
$\lambda_k\geq 0$. We want to show that $\lambda_1>0$. Since the
$\lambda_k$ are solutions to (\ref{eq:D3}), $\lambda_1=0$ if and
only if the constant term of (\ref{eq:D3}) is zero. This constant
term is
\begin{equation*}
B_2^2\left(B_1^2-4B_2\right)=a^2(1-c)^2D_2>0
\end{equation*}
where $D_2$ is given by (\ref{eq:D2}). This concludes the proof.
\end{proof}
Let us now compute the probability density $\underline{F}$ when it
is supported on 2 disjoint intervals.
\begin{theorem}\label{thm:density}
Suppose $\Delta>0$. Then the p.d.f $\underline{F}$ is supported on
$[\lambda_1,\lambda_2]\cup[\lambda_3,\lambda_4]$ with the following
density $d\underline{F}(z)=\rho(z)dz$
\begin{equation}\label{eq:density}
\frac{3}{2\pi}\left|\left(\frac{r(z)+\sqrt{-\frac{1}{27a^4z^4}D_3(z)}}{2}\right)^{\frac{1}{3}}-\left(\frac{r(z)-\sqrt{-\frac{1}{27a^4z^4}D_3(z)}}{2}\right)^{\frac{1}{3}}\right|,
\end{equation}
where $D_3(z)$ and $r(z)$ are given by
\begin{equation*}
\begin{split}
D_3(z)&=(1-a)^2\prod_{j=1}^4(z-\lambda_j),\\
r(z)&=\frac{1}{27}\Bigg(-\frac{2B_2^3}{a^3}z^{-3}+\left(\frac{9B_1B_2}{a^2}-\frac{6A_2B_2^2}{a^3}\right)z^{-2}+\left(\frac{9B_2}{a^2}+\frac{9B_1A_2}{a^2}-\frac{27}{a}-\frac{6A_2^2B_2}{a^3}\right)z^{-1}\\
&+\left(\frac{9A_2}{a^2}-\frac{2A_2^3}{a^3}\right)\Bigg).
\end{split}
\end{equation*}
The cube root in (\ref{eq:density}) is chosen such that
$\sqrt[3]{A}\in\mathbb{R}$ for $A\in\mathbb{R}$ and the square root
is chosen such that $\sqrt{A}>0$ for $A>0$.
\end{theorem}
\begin{proof} Let
$z_0\in[\lambda_1,\lambda_2]\cup[\lambda_3,\lambda_4]$ be a point in
the support on $\underline{F}$. Then as a polynomial in $\xi$,
(\ref{eq:curve2}) has the following solutions at $z=z_0$.
\begin{equation}\label{eq:cubsol}
\begin{split}
\xi_{R}(z_0)&=-\frac{1}{3}\left(\frac{A_2}{a}+\frac{B_2}{az_0}\right)+\sqrt[3]{u_+}+\sqrt[3]{u_-},\\
\xi_{I_1}(z_0)&=-\frac{1}{3}\left(\frac{A_2}{a}+\frac{B_2}{az_0}\right)+\omega\sqrt[3]{u_+}+\omega^2\sqrt[3]{u_-},\\
\xi_{I_2}(z_0)&=-\frac{1}{3}\left(\frac{A_2}{a}+\frac{B_2}{az_0}\right)+\omega^2\sqrt[3]{u_+}+\omega\sqrt[3]{u_-},
\end{split}
\end{equation}
where $u_{\pm}$ are
\begin{equation}\label{eq:u}
u_{\pm}=\frac{r(z)\pm\sqrt{-\frac{1}{27a^4z^4}D_3(z)}}{2}
\end{equation}
and $\omega=e^{\frac{2\pi i}{3}}$ is the cube root of unity. The
functions $r(z)$, $D_3(z)$ and constants $A_2$, $B_1$, $B_2$ defined
in the statement of the theorem. The branches of the cube root in
(\ref{eq:cubsol}) and the square root in (\ref{eq:u}) can be chosen
arbitrarily, as long as all the branches are the same. We will not
treat (\ref{eq:cubsol}) as analytic functions in $z_0$, but merely
consider them as the different values of the roots of
(\ref{eq:curve2}) at the point $z=z_0$.

Since $D_3(z)<0$ on $\mathrm{supp}(\underline{F})$,
$u_{\pm}(z)\in\mathbb{R}$ on the support of $\underline{F}$. Let us
choose the square root and cubic root in (\ref{eq:cubsol}) in the
way that is indicated in the statement of the theorem. Then for any
point $z_0\in\mathrm{supp}(\underline{F})$, $\xi_R(z_0)$ is real
while $\xi_{I_1}(z_0)$ and $\xi_{I_2}(z_0)$ are complex conjugate to
each other. Since the analytic continuation of
$m_{\underline{F}}(z)$ on the positive side of $\mathbb{R}$ becomes
one of the complex roots on $\mathrm{supp}(\underline{F})$, we see
that the imaginary part of $m_{\underline{F}}(z_0)$ must coincide
with the imaginary part of either $\xi_{I_1}(z_0)$ or
$\xi_{I_2}(z_0)$. In particular, we have
\begin{equation}\label{eq:imf}
\left|\mathrm{Im}\left(m_{\underline{F}}(z_0)\right)\right|=\frac{1}{2}\left|\xi_{I_1}(z_0)-\xi_{I_2}(z_0)\right|.
\end{equation}
From the inversion formula and the fact that the p.d.f.
$\underline{F}$ is non-negative, we see that the probability density
for $\underline{F}$ is given by
\begin{equation*}
d\underline{F}=\frac{1}{2\pi}\left|\xi_{I_1}(z)-\xi_{I_2}(z)\right|dz,\quad
z\in\mathrm{supp}(\underline{F}).
\end{equation*}
By substituting (\ref{eq:cubsol}) back into this equation, we
arrived at the theorem.
\end{proof}
\subsubsection{Branch cuts of Riemann surface}\label{se:branch}

In this section we will show that the solution $\xi_1(z)$
(\ref{eq:xiinfty}) of (\ref{eq:curve2}) coincides with the Stieltjes
transform $m_{\underline{F}}(z)$ in $\mathbb{C}^+\cup\mathbb{R}$.
Furthermore, when the support of $\underline{F}$ consists of 2
disjoint intervals, the function $\xi_1(z)$ and hence
$m_{\underline{F}}(z)$ will not be analytic at any of the points
$\lambda_k$, $k=1,\ldots,4$.

The solutions $\xi_j(z)$ of (\ref{eq:curve2}) will not be analytic
at the branch point $(\lambda_k,\gamma_k)$ if
$\xi_j(\lambda_k)=\gamma_k$. Since all the $\lambda_k$ are on the
real axis, and that the only possible pole of these functions is at
$z=0$, there exists analytic continuations of the $\xi_j(z)$ in
$\mathbb{C}^+$ that are continuous up to $\mathbb{R}\setminus\{0\}$.
We have the following lemma
\begin{lemma}\label{le:bp}
Let $\xi_j(z)$ be the solutions of (\ref{eq:curve2}) in
$\left(\mathbb{C}^+\cup\mathbb{R}\right)\setminus\{0\}$ that has the
asymptotic behavior (\ref{eq:xiinfty}). Then for $j=1,2,3$,
$\xi_j(\lambda_1)=\gamma_1$ if and only if
$\xi_j(\lambda_2)=\gamma_2$. Similarly, $\xi_j(\lambda_3)=\gamma_3$
if and only if $\xi_j(\lambda_4)=\gamma_4$.
\end{lemma}
\begin{proof} Suppose $\xi_j(\lambda_1)=\gamma_1$ for some $j$. Then
$\xi_j$ takes the value of the double root of (\ref{eq:curve2}) at
$\lambda_1$ and by continuity, $\xi_j$ becomes one of the complex
roots of (\ref{eq:curve2}) on the interval $[\lambda_1,\lambda_2]$.
This implies that at $\lambda_2$, $\xi_j$ will take the value of the
double root of (\ref{eq:curve2}) again. Therefore we have
$\xi_j(\lambda_2)=\gamma_2$. Similarly, we see that if
$\xi_j(\lambda_2)=\gamma_2$, then we also have
$\xi_j(\lambda_1)=\gamma_1$. Hence $\xi_j(\lambda_1)=\gamma_1$ if
and only if $\xi_j(\lambda_2)=\gamma_2$. Applying the same argument
to the points $\lambda_3$ and $\lambda_4$, the lemma is proven.
\end{proof}
Let us now show that the function $\xi_1(z)$ is in fact the
Stieltjes transform $m_{\underline{F}}(z)$. From the asymptotic
behavior of the $\xi_j(z)$ (\ref{eq:xiinfty}) and the fact that the
Stieltjes transform $m_{\underline{F}}(z)$ solves (\ref{eq:curve2})
and vanishes as $z\rightarrow\infty$, we see that
\begin{equation}\label{eq:stiexi}
m_{\underline{F}}(z)=\xi_1(z),\quad
z\in\left(\mathbb{C}^+\cup\mathbb{R}\right)\setminus\{0\}.
\end{equation}
For $c<1$, it was shown in \cite{CS} that $F$ has a continuous
density and hence the Stieltjes transform $m_F(z)$ does not have any
poles. Therefore, by (\ref{eq:stielim}), we see that
$m_{\underline{F}}(z)$, and hence $\xi_1(z)$, has the following
singularity at $z=0$.
\begin{equation}\label{eq:xizero}
\xi_1(z)=-\frac{1-c}{z}+O(1),\quad z\rightarrow 0.
\end{equation}
We will now show that for all $\lambda_k$, we have
$\xi_1(\lambda_k)=\gamma_k$.
\begin{proposition}\label{pro:branchp}
Let $\xi_1(z)$ be the solution of (\ref{eq:curve2}) in
$\left(\mathbb{C}^+\cup\mathbb{R}\right)\setminus\{0\}$ with the
asymptotic behavior indicated as in (\ref{eq:xiinfty}). Then we have
$\xi_1(\lambda_k)=\gamma_k$ for $k=1,\ldots, 4$. In particular,
$\xi_1(z)$ is not analytic at any of the points $\lambda_k$.
\end{proposition}
\begin{proof} Suppose for some $\lambda_k$ we have
$\xi_1(\lambda_k)\neq\gamma_k$ and that the support of
$\underline{F}$ is on the left hand side of $\lambda_k$. Then we
have $z^{\prime}\left(\xi_1(\lambda_k)\right)\neq 0$. By lemma
\ref{le:CS}, we must have
$z^{\prime}\left(m_{\underline{F}}(\lambda_k+\delta)\right)>0$ for
small enough $\delta>0$ as $\lambda_k+\delta$ does not belong to the
support of $\underline{F}$ and that $m_{\underline{F}}$ is real on
$\lambda_k+\delta$. Then by continuity, we have
$z^{\prime}\left(m_{\underline{F}}(\lambda_k)\right)>0$. (As we
assume $z^{\prime}\left(m_{\underline{F}}(\lambda_k)\right)\neq0$)

Since $\xi_1(z)$ does not coincide with the double root of
(\ref{eq:curve2}) at $\lambda_k$, it must be real in
$[\lambda_k-\epsilon,\lambda_k+\epsilon]$ for some small
$\epsilon>0$. For small enough $\epsilon$, we can assume that
neither 0, -1 or $-\frac{1}{a}$ are inside this interval.

Therefore in the interval $[\lambda_k-\epsilon,\lambda_k+\epsilon]$,
the Stieltjes transform $m_{\underline{F}}(z)$ is real and by
continuity, $z^{\prime}\left(m_{\underline{F}}\right)>0$ inside this
interval. Since this interval consists of points that belong to the
support of $\underline{F}$, this would imply that for some point
$z_0\in\mathrm{supp}(\underline{F})$, we have
$z^{\prime}\left(m_{\underline{F}}\right)>0$ and
$m_{\underline{F}}\in\mathbb{R}\setminus\{0,-1,-\frac{1}{a}\}$. This
contradicts lemma \ref{le:CS} and hence $\xi_1(\lambda_k)=\gamma_k$.
By using exactly the same argument, we can prove the proposition for
$\lambda_k$ when the support lies on the right hand side of
$\lambda_k$.
\end{proof}
This proposition implies that $\xi_1(z)$ is not analytic at the
points $\lambda_k$ for $k=1,\ldots,4$. With this information, we can
now find out the structure of the Riemann surface $\Lie$.
\begin{proposition}\label{pro:sheet}
If $a>1$, then $\xi_3(z)$ has branch points at $\lambda_3$ and
$\lambda_4$ and $\xi_2(z)$ has branch points at $\lambda_1$ and
$\lambda_2$. On the other hand, if $a<1$, then $\xi_3(z)$ has branch
points at $\lambda_1$ and $\lambda_2$ while $\xi_2(z)$ has branch
points at $\lambda_3$ and $\lambda_4$.
\end{proposition}
\begin{proof} We will prove the statement for the case $a>1$, the
case $a<1$ can be proven by similar argument. Suppose $a>1$, then
for large enough $z_0>\lambda_4$, all the functions $\xi_j$ are real
for $z\in\mathbb{R}$ and we have $\xi_1(z_0)>\xi_3(z_0)>\xi_2(z_0)$
by (\ref{eq:xiinfty}). This ordering must preserve at $\lambda_4$ as
the roots cannot coincide between $z_0$ and $\lambda_4$. At
$\lambda_4$, one of the roots must coincide with $\xi_1$ and from
the ordering $\xi_1(z_0)>\xi_3(z_0)>\xi_2(z_0)$, we must have
$\xi_1(\lambda_4)=\xi_3(\lambda_4)=\gamma_4$. By lemma \ref{le:bp},
we also have $\xi_3(\lambda_3)=\gamma_3$. On the other hand, for
small enough $\lambda_1>\epsilon>0$, all three $\xi_j$ will be real.
From the asymptotic behavior of $\xi_1(z)$ at $0$ (\ref{eq:xizero})
we see that for small enough $\epsilon$, $\xi_1(\epsilon)$ will be
smaller than both $\xi_2(\epsilon)$ and $\xi_3(\epsilon)$. From the
asymptotic behavior of $\xi_2$ and $\xi_3$ at $z=-\infty$
(\ref{eq:xiinfty}) and the fact that these 2 functions has no
singularity and cannot coincide on $(-\infty, \lambda_1)$, we see
that at $z=\epsilon$, we must have
$\xi_3(\epsilon)>\xi_2(\epsilon)$. Therefore we have
$\xi_3(\epsilon)>\xi_2(\epsilon)>\xi_1(\epsilon)$. This ordering
must again be preserved at $\lambda_1$. Therefore we must have
$\xi_1(\lambda_1)=\xi_2(\lambda_1)=\gamma_1$. By lemma \ref{le:bp},
we also have $\xi_2(\lambda_2)=\gamma_2$. This shows that for $a>1$,
the branch points of $\xi_3$ are $\lambda_3$ and $\lambda_4$ while
the branch points of $\xi_2$ are $\lambda_1$ and $\lambda_2$. The
proof for the case $a<1$ follows from similar argument.
\end{proof}
We will now define the branch cuts of the function $\xi_1(z)$ to be
$[\lambda_1,\lambda_2]\cup[\lambda_3,\lambda_4]$ and the branch cut
for $\xi_2(z)$, $\xi_3(z)$ to be $[\lambda_1,\lambda_2]$ and
$[\lambda_3,\lambda_4]$ respectively for $a>1$ and
$[\lambda_3,\lambda_4]$, $[\lambda_1,\lambda_2]$ respectively for
$a<1$. We then have the following relations between the $\xi_j$ on
the branch cuts
\begin{equation}\label{eq:bound}
\begin{split}
&1. \quad a>1,\quad \xi_{1,\pm}(z)=\left\{
               \begin{array}{ll}
                 \xi_{2,\mp}(z), & \hbox{$z\in[\lambda_1,\lambda_2]$;} \\
                 \xi_{3,\mp}(z), & \hbox{$z\in[\lambda_3,\lambda_4]$.}
               \end{array}
             \right.\\
&2. \quad a<1,\quad \xi_{1,\pm}(z)=\left\{
               \begin{array}{ll}
                 \xi_{2,\mp}(z), & \hbox{$z\in[\lambda_3,\lambda_4]$;} \\
                 \xi_{3,\mp}(z), & \hbox{$z\in[\lambda_1,\lambda_2]$.}
               \end{array}
             \right.
\end{split}
\end{equation}
where $\xi_{j,\pm}(z)$ indicates the boundary values of $\xi_j$ on
the $\pm$ sides of the branch cuts.

The branch cut structure of the Riemann surface $\Lie$ is indicated
in Figure \ref{fig:sheets}.
\begin{figure}
\centering
\includegraphics[scale=0.75]{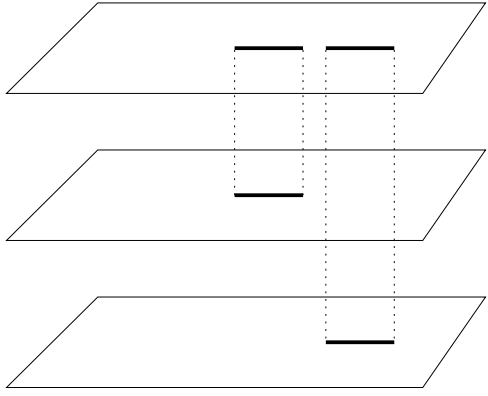} \caption{The branch cut
structure of the Riemann surface $\Lie$ when
$a>1$.}\label{fig:sheets}
\end{figure}

We will now define the functions $\theta_j(z)$ to be the the
integrals of $\xi_j(z)$.
\begin{equation}\label{eq:theta}
\begin{split}
\theta_1(z)&=\int_{\lambda_4}^z\xi_1(x)dx,\quad
\theta_2(z)=\int_{\lambda_{k_2}}^z\xi_2(x)dx,\quad
\theta_3(z)=\int_{\lambda_{k_3}}^z\xi_3(x)dx,\\
k_2&=2,\quad k_3=4,\quad a>1,\quad k_2=4,\quad k_3=2,\quad a<1.
\end{split}
\end{equation}
where the integration paths of the above integrals are chosen such
that they do not intersect the real axis, except perhaps at the end
points.

Then from (\ref{eq:xiinfty}), (\ref{eq:zeroasym}) and
(\ref{eq:xizero}), we see that the integrals (\ref{eq:theta}) have
the following behavior at $z=\infty$ and $z=0$.
\begin{equation}\label{eq:asymtheta}
\begin{split}
\theta_1(z)&=-\log z+l_1+O\left(z^{-1}\right),\quad
z\rightarrow\infty,\\
\theta_1(z)&=-(1-c)\log z+O\left(1\right),\quad
z\rightarrow 0,\\
\theta_2(z)&=-z+c(1-\beta)\log z+l_2+O\left(z^{-1}\right),\quad
z\rightarrow\infty,\\
\theta_2(z)&=O\left(1\right),\quad
z\rightarrow 0,\\
\theta_3(z)&=-\frac{z}{a}+c\beta\log
z+l_3+O\left(z^{-1}\right),\quad
z\rightarrow\infty,\\
\theta_3(z)&=O\left(1\right),\quad z\rightarrow 0.
\end{split}
\end{equation}
for some constants $l_1$, $l_2$ and $l_3$.

From the behavior of $\xi_j(z)$ on the cuts, we have the following
analyticity properties of the $\theta_j(z)$.
\begin{lemma}\label{le:cuttheta}
The integral $\theta_1(z)$ is analytic on
$\mathbb{C}\setminus(-\infty,\lambda_4]$ and continuous up to
$\mathbb{R}\setminus\{0\}$. The integrals $\theta_2(z)$ and
$\theta_3(z)$ are analytic on
$\mathbb{C}\setminus(-\infty,\lambda_{k_2}]$ and
$\mathbb{C}\setminus(-\infty,\lambda_{k_3}]$ respectively and are
continuous up to $\mathbb{R}$. Across the real axis, they have the
following jump discontinuities.
\begin{equation}\label{eq:cuttheta}
\begin{split}
\theta_{1,\pm}(z)&=\theta_{j,\mp}(z)+\theta_{1,\pm}(\lambda_{k_j}),\quad z\in [\lambda_{{k_j}-1},\lambda_{k_j}],\quad j=1,2,\\
\theta_{1,+}(z)&=\theta_{1,-}(z)-2c\beta\pi i,\quad z\in
[\lambda_2,\lambda_3],\quad a>1,\\
\theta_{1,+}(z)&=\theta_{1,-}(z)-2c(1-\beta)\pi i,\quad z\in
[\lambda_2,\lambda_3],\quad a<1,\\
\theta_{1,+}(z)&=\theta_{1,-}(z)-2c\pi i,\quad z\in
(0,\lambda_1],\\
\theta_{1,+}(z)&=\theta_{1,-}(z)-2\pi i,\quad z\in
(-\infty,0],\\
\theta_{2,+}(z)&=\theta_{2,-}(z)+2c(1-\beta)\pi i,\quad z\in
(-\infty,\lambda_{{k_2}-1}],\\
\theta_{3,+}(z)&=\theta_{3,-}(z)+2c\beta\pi i,\quad z\in
(-\infty,\lambda_{k_3-1}].
\end{split}
\end{equation}
\end{lemma}
\begin{proof} From the jump discontinuities of $\xi_j(z)$ in
(\ref{eq:bound}), we see that
\begin{equation*}
\begin{split}
\left(\int_{\lambda_{k_j}}^z\xi_{1}(x)dx\right)_{\pm}&=\theta_{j,\mp}(z),\quad z\in [\lambda_{{k_j}-1},\lambda_{k_j}],\quad j=2,3,\\
\end{split}
\end{equation*}
By comparing this and (\ref{eq:theta}), we obtain the first equation
in (\ref{eq:cuttheta}).

Let us now show that
\begin{equation}\label{eq:theta2}
\theta_{2,+}(z)=\theta_{2,-}(z)+2c(1-\beta)\pi i,\quad z\in
(-\infty,\lambda_{{k_2}-1}].
\end{equation}
The corresponding equation for $\theta_1(z)$ and $\theta_3(z)$ can
be proven in a similar way.

Let $z\in(-\infty,\lambda_{{k_2}-1}]$ and let $\Gamma_{\pm}$ be
contours from $\lambda_{{k_2}}$ to $z$ such that $\Gamma_{\pm}$ are
in the upper and lower half planes respectively. Then we have
\begin{equation*}
\theta_{2,\pm}(z)=\int_{\Gamma_{\pm}}\xi_2(x)dx.
\end{equation*}
Let $\Gamma$ be the close loop on $\mathbb{C}$ such that
$\Gamma=\Gamma_+-\Gamma_-$, then we have
\begin{equation*}
\theta_{2,+}(z)=\theta_{2,-}(z)+\oint_{\Gamma}\xi_2(x)dx,\quad z\in
(-\infty,\lambda_{{k_2}-1}].
\end{equation*}
By Cauchy's theorem, we can deform the loop $\Gamma$ such that
$\Gamma$ becomes a loop around $z=\infty$. By computing the residue,
we arrived at (\ref{eq:theta2}).

On the other hand, we can also deform $\Gamma$ so that it becomes a
loop enclosing the branch cut $[\lambda_{{k_2}-1},\lambda_{k_2}]$ of
$\xi_2(z)$. Then we have
\begin{equation}\label{eq:theta2den}
\oint_{\Gamma}\xi_2(x)dx=-\int_{\lambda_{{k_2}-1}}^{\lambda_{k_2}}\left(\xi_{2,+}(x)-\xi_{2,-}(x)\right)dx.
\end{equation}
From (\ref{eq:bound}), we have $\xi_{2,-}(z)=\xi_{1,+}(z)$. Since
$\xi_{2}(z)$ and $\xi_1(z)$ are the 2 complex roots of
(\ref{eq:curve2}) on $[\lambda_{k_2-1},\lambda_{k_2}]$, we have
\begin{equation*}
\xi_{2,+}(x)-\xi_{2,-}(x)=\xi_{1,-}(x)-\xi_{1,+}(x)=-2i\mathrm{Im}\left(\xi_{1,+}(x)\right),\quad
x\in[\lambda_{{k_2}-1},\lambda_{{k_2}}].
\end{equation*}
By comparing (\ref{eq:theta2den}) and (\ref{eq:theta2}), we obtain
\begin{equation}\label{eq:theta2mass}
\int_{\lambda_{{k_2}-1}}^{\lambda_{{k_2}}}\mathrm{Im}\left(\xi_{1,+}(x)\right)dx=c(1-\beta)\pi.
\end{equation}
By using similar argument, we see that
\begin{equation}\label{eq:theta3mass}
\begin{split}
&\theta_{3,+}(z)=\theta_{3,-}(z)+2c\beta\pi i,\quad z\in
(-\infty,\lambda_{k_3-1}].\\
&\int_{\lambda_{k_3-1}}^{\lambda_{k_3}}\mathrm{Im}\left(\xi_{1,+}(x)\right)dx=c\beta\pi,\\
&\theta_{1,+}(z)=\theta_{1,-}(z)-2i\int_{\lambda_{3}}^{\lambda_{4}}\mathrm{Im}\left(\xi_{1,+}(x)\right)dx
,\quad z\in [\lambda_2,\lambda_3],\\
&\theta_{1,+}(z)=\theta_{1,-}(z)-2c\pi i,\quad z\in
(0,\lambda_1],\\
&\theta_{1,+}(z)=\theta_{1,-}(z)-2\pi i,\quad z\in (-\infty,0].
\end{split}
\end{equation}
From these, we have
\begin{equation}\label{eq:thetamass}
\begin{split}
\int_{\lambda_{{k_2}-1}}^{\lambda_{{k_2}}}\mathrm{Im}\left(\xi_{1,+}(x)\right)dx=c(1-\beta)\pi,\quad
\int_{\lambda_{k_3-1}}^{\lambda_{k_3}}\mathrm{Im}\left(\xi_{1,+}(x)\right)dx=c\beta\pi.
\end{split}
\end{equation}
and hence
\begin{equation*}
\begin{split}
\theta_{1,+}(z)&=\theta_{1,-}(z)-2i\int_{\lambda_{3}}^{\lambda_{4}}\mathrm{Im}\left(\xi_{1,+}(x)\right)dx\\
&=\theta_{1,-}(z)-\left\{
                    \begin{array}{ll}
                      2c\beta\pi i, & \hbox{$a>1$;} \\
                      2c(1-\beta)\pi i, & \hbox{$a<1$.}
                    \end{array}
                  \right.
\end{split}
\end{equation*}
for $z\in[\lambda_2,\lambda_3]$. This proves the lemma.
\end{proof}
From the proof of this lemma, we see that the mass of
$\underline{F}$ is distributed between $[\lambda_1,\lambda_2]$ and
$[\lambda_3,\lambda_4]$ in the following way.
\begin{corollary}\label{cor:mass}
Let $k_2=2$, $k_3=4$ for $a>1$ and $k_2=4$, $k_3=2$ for $a<1$. Then
the mass of the measure $\underline{F}$ is distributed between
$[\lambda_1,\lambda_2]$ and $[\lambda_3,\lambda_4]$ in the following
way.
\begin{equation}\label{eq:mass}
\underline{F}\left([\lambda_{{k_2}-1},\lambda_{{k_2}}]\right)=c(1-\beta),\quad
\underline{F}\left([\lambda_{k_3-1},\lambda_{k_3}]\right)=c\beta.
\end{equation}
\end{corollary}
We will conclude this section with the following results on the
relative sizes of the $\mathrm{Re}\left(\theta_j(z)\right)$, which
are essential in the implementation of the Riemann-Hilbert method.
\begin{lemma}\label{le:size}
The real parts of $\theta_j(z)$ are continuous across $\mathbb{R}$.
Let us denote the real parts by $\mathrm{Re}(\theta_j(z))$. Then we
have the following
\begin{equation}\label{eq:size}
\begin{split}
\mathrm{Re}(\theta_1(z))&-\mathrm{Re}(\theta_1(\lambda_{k_j}))>\mathrm{Re}(\theta_j(z)),\quad
x\in\mathbb{R}_+\setminus[\lambda_{{k_j}-1},\lambda_{k_j}] ,\quad
j=1,2.
\end{split}
\end{equation}
\end{lemma}
\begin{proof} First note that, by (\ref{eq:bound}), we see that
$\xi_{j,\pm}(z)$ are complex conjugations on
$[\lambda_1,\lambda_2]\cup[\lambda_3,\lambda_4]$. Therefore the real
part $\mathrm{Re}(\xi_j(z))$ of $\xi_j(z)$ and hence
$\mathrm{Re}(\theta_j(z))$ is continuous across the real axis.

From the proof of Proposition \ref{pro:sheet}, we have, for $j=2,3$,
\begin{equation}\label{eq:ineq0}
\begin{split}
\xi_1(z)&>\xi_j(z),\quad
z\in[\lambda_{4},\infty),\\
\xi_1(z)&<\xi_j(z),\quad z\in(0,\lambda_{1}].
\end{split}
\end{equation}
Since $\xi_1(z)$ and $\xi_j(z)$ are complex conjugate solutions on
$[\lambda_{k_j-1},\lambda_{k_j}]$ and that $\xi_1(z)$ is the branch
with positive imaginary part (recall that
$\xi_i(z)=m_{\underline{F}}(z)$ and from (\ref{eq:inver}), we see
that the imaginary part of $m_{\underline{F}}(z)$ is positive on
$\mathrm{supp}(\underline{F})$), we have, by (\ref{eq:cubsol}), the
following behavior near the branch points
\begin{equation}\label{eq:branchbe}
\begin{split}
\xi_1(z)&=\gamma_{k_j-1}+i\nu_1^j\sqrt{z-\lambda_{k_j-1}}+O\left((z-\lambda_{k_j-1})\right),\quad z\rightarrow\lambda_{k_j-1},\\
\xi_j(z)&=\gamma_{k_j-1}-i\nu_1^j\sqrt{z-\lambda_{k_j-1}}+O\left((z-\lambda_{k_j-1})\right),\quad z\rightarrow\lambda_{k_j-1},\\
\xi_1(z)&=\gamma_{k_j}+\nu_2^j\sqrt{z-\lambda_{k_j}}+O\left((z-\lambda_{k_j})\right),\quad z\rightarrow\lambda_{k_j},\\
\xi_j(z)&=\gamma_{k_j}-\nu_2^j\sqrt{z-\lambda_{k_j}}+O\left((z-\lambda_{k_j})\right),\quad
z\rightarrow\lambda_{k_j}.
\end{split}
\end{equation}
where $\nu_1^j$ and $\nu_2^j$ are real constants and the branch cut
of the square root in the first 2 equations is chosen to be the
positive real axis, with the branch chosen such that
$\nu_1^j\sqrt{z-\lambda_{k_j-1}}$ is positive on the left hand side
of the positive real axis; while the branch cut of the square root
in the last 2 equations is chosen to be the negative real axis, with
the branch chosen such that $\nu_2^j\sqrt{z-\lambda_{k_j-1}}$ is
positive on the positive real axis. From (\ref{eq:branchbe}), we see
that there exists $\delta>0$ such that
\begin{equation}\label{eq:ineq01}
\begin{split}
\mathrm{Re}\left(\xi_1(z)\right)&>\mathrm{Re}\left(\xi_j(z)\right),\quad
z\in(\lambda_{k_j},\lambda_{k_j}+\delta),\\
\mathrm{Re}\left(\xi_1(z)\right)&<\mathrm{Re}\left(\xi_j(z)\right),\quad
z\in(\lambda_{k_j-1}-\delta,\lambda_{k_j-1}).
\end{split}
\end{equation}
The ordering (\ref{eq:ineq01}) must be preserved until $z$ hits
another branch cut. Therefore from (\ref{eq:ineq0}) and
(\ref{eq:ineq01}), we have
\begin{equation}\label{eq:ineq}
\begin{split}
\mathrm{Re}\left(\xi_1(z)\right)&>\mathrm{Re}\left(\xi_j(z)\right),\quad
z\in[\lambda_{k_j},\infty)\setminus\mathrm{supp}(\underline{F}),\\
\mathrm{Re}\left(\xi_1(z)\right)&<\mathrm{Re}\left(\xi_j(z)\right),\quad
z\in(0,\lambda_{{k_j}-1}]\setminus\mathrm{supp}(\underline{F}).
\end{split}
\end{equation}
We will now show that these inequalities hold for the whole
intervals $(0,\lambda_{k_j-1}]$ and $[\lambda_{k_j},\infty)$. Let
$l\neq j$. Suppose either of the following happens in
$[\lambda_{k_l-1},\lambda_{k_l}]$,
\begin{equation}\label{eq:vio}
\begin{split}
\mathrm{Re}\left(\xi_j(z)-\xi_1(z)\right)&>0,\quad \textrm{if
$\lambda_{k_l}>\lambda_{k_j}$},\\
\mathrm{Re}\left(\xi_j(z)-\xi_1(z)\right)&<0,\quad \textrm{if
$\lambda_{k_l}<\lambda_{k_j}$}.
\end{split}
\end{equation}
Then by (\ref{eq:ineq}), there must be a point
$z_1\in[\lambda_{k_l-1},\lambda_{k_l}]$ such that
\begin{equation}\label{eq:equal}
\begin{split}
\mathrm{Re}(\xi_j(z_1))=\mathrm{Re}(\xi_1(z_1)),\\
\frac{d}{dz}\left(\mathrm{Re}\left(\xi_j(z_1)-\xi_1(z_1)\right)\right)\geq
0.
\end{split}
\end{equation}
The values of the $\xi(z)$ at the point $z_1$ are given by
(\ref{eq:cubsol}) with $\xi_j(z)=\xi_R(z)$ while $\xi_1(z)$ is one
of the complex roots $\xi_{I_1}(z)$ or $\xi_{I_2}(z)$. Taking the
difference between the real parts of $\xi_R(z)$ and $\xi_{I_1}(z)$
(or $\xi_{I_2}(z)$ which has the same real part as $\xi_{I_1}(z)$),
we have
\begin{equation}\label{eq:realdiff}
\mathrm{Re}\left(\xi_j(z)-\xi_{1}(z)\right)=2\left(\sqrt[3]{u_+}+\sqrt[3]{u_-}\right)
\end{equation}
Since $z_1$ is a point on $\mathrm{supp}(\underline{F})$, from Lemma
\ref{le:CS}, we see that the derivative $\xi_R^{\prime}(z_1)$ of the
real root at $z_1$ must be non-positive. By (\ref{eq:cubsol}), this
derivative is given by
\begin{equation}\label{eq:diffreal}
\xi_{R}^{\prime}(z_0)=\frac{1}{3}\frac{B_2}{az_1^2}+\frac{d}{dz}\left(\sqrt[3]{u_+}+\sqrt[3]{u_-}\right).
\end{equation}
By (\ref{eq:realdiff}) and (\ref{eq:diffreal}), we see that if
$\frac{d}{dz}\left(\mathrm{Re}\left(\xi_j(z_1)-\xi_1(z_1)\right)\right)\geq
0$, then (\ref{eq:diffreal}) will be positive, this contradicts
Lemma \ref{le:CS} and hence
$\mathrm{Re}\left(\xi_j(z)-\xi_{1}(z)\right)$ cannot vanish in
$[\lambda_{k_l-1},\lambda_{k_l}]$. This, together with
(\ref{eq:ineq}) implies
\begin{equation}\label{eq:ineq2}
\begin{split}
\mathrm{Re}\left(\xi_1(z)\right)&>\mathrm{Re}\left(\xi_j(z)\right),\quad
z\in[\lambda_{k_j},\infty),\\
\mathrm{Re}\left(\xi_1(z)\right)&<\mathrm{Re}\left(\xi_j(z)\right),\quad
z\in(0,\lambda_{{k_j}-1}].
\end{split}
\end{equation}
Therefore we have
\begin{equation*}
\begin{split}
&\int_{\lambda_{k_j}}^z\mathrm{Re}\left(\xi_1(x)-\xi_j(x)\right)dx>0,\quad
z>\lambda_{k_j},\\
&\int_{\lambda_{{k_j}-1}}^z\mathrm{Re}\left(\xi_1(x)-\xi_j(x)\right)dx>0,\quad
0<z<\lambda_{{k_j}-1},
\end{split}
\end{equation*}
Since $\xi_1(z)-\xi_j(z)$ is purely imaginary on
$[\lambda_{{k_j}-1},\lambda_{k_j}]$, we have
\begin{equation}\label{eq:ima}
\int_{\lambda_{{k_j}-1}}^{\lambda_{{k_j}}}\mathrm{Re}\left(\xi_1(x)-\xi_j(x)\right)dx=0,
\end{equation}
Then from the definition (\ref{eq:theta}) of the $\theta_j(z)$, we
obtain (\ref{eq:size}).
\end{proof}
The final result in this section concerns about the behavior of
these real parts in a neighborhood of the branch cuts.
\begin{lemma}\label{le:lens}
The open interval $(\lambda_{{k_j}-1},\lambda_{k_j})$ , $j=2,3$ has
a neighborhood $U_j$ in the complex plane such that
\begin{equation}\label{eq:lens}
\mathrm{Re}(\theta_j(z))+\mathrm{Re}\left(\theta_1(\lambda_{k_j})\right)>\mathrm{Re}(\theta_1(z))>
\mathrm{Re}(\theta_l(z))+\mathrm{Re}\left(\theta_1(\lambda_{k_l})\right),
\end{equation}
where $l=2,3$ and $l\neq j$.
\end{lemma}
\begin{proof}Since $\xi_1(z)-\xi_j(z)$ is purely imaginary on
$[\lambda_{{k_j}-1},\lambda_{k_j}]$, we have
\begin{equation}\label{eq:boundvalue}
\int_{\lambda_{{k_j}}}^{z}\mathrm{Re}\left(\xi_1(x)-\xi_j(x)\right)dx=\mathrm{Re}\left(\theta_1(z)-\theta_1(\lambda_{k_j})-\theta_j(z)\right)=0,\quad
z\in [\lambda_{k_j-1},\lambda_{k_j}],
\end{equation}
On the positive and negative sides of
$(\lambda_{k_j-1},\lambda_{k_j})$, the derivative of the functions
$\theta_{1,\pm}(z)-\theta_{j,\pm}(z)$ are given by
$\xi_{1,\pm}(z)-\xi_{j,\pm}(z)$ and are purely imaginary. In fact,
since $\xi_{1,+}(z)=m_{\underline{F}}$, we see that
$\xi_{1,+}(z)-\xi_{j,-}(z)=2\pi i\rho(z)$ where $\rho(z)>0$ is the
density function of $\underline{F}$. On the other hand, by the jump
discontinuities (\ref{eq:bound}), we see that
$\xi_{1,-}(z)-\xi_{j,+}(z)=-2\pi i\rho(z)$. Hence by the Cauchy
Riemann equation, the real part of
$\theta_1(z)-\theta_j(z)-\theta_1(\lambda_{k_j})$ is decreasing as
we move from $(\lambda_{k_j-1},\lambda_{k_j})$ into the upper half
plane. From (\ref{eq:boundvalue}), we see that
$\mathrm{Re}(\theta_1(z)-\theta_j(z)-\theta_1(\lambda_{k_j}))<0$ for
$z$ in the upper half plane near $(\lambda_{k_j-1},\lambda_{k_j})$.
Similarly, we also have
$\mathrm{Re}(\theta_1(z)-\theta_j(z)-\theta_1(\lambda_{k_j}))<0$ for
$z$ in the lower half plane near $(\lambda_{k_j-1},\lambda_{k_j})$.
Therefore in a neighborhood of $(\lambda_{k_j-1},\lambda_{k_j})$, we
have
\begin{equation}\label{eq:size1}
\mathrm{Re}(\theta_1(z))<\mathrm{Re}(\theta_j(z)+\theta_1(\lambda_{k_j})).
\end{equation}
Now by Lemma \ref{le:size}, we see that, if $l=2,3$ and $l\neq j$,
then we have
\begin{equation}\label{eq:s1}
\mathrm{Re}(\theta_1(z))>\mathrm{Re}(\theta_l(z)+\theta_1(\lambda_{k_l})),\quad
z\in(\lambda_{k_j-1},\lambda_{k_j}).
\end{equation}
From (\ref{eq:s1}) and (\ref{eq:size1}), we see that (\ref{eq:lens})
is true.
\end{proof}
\section{Riemann-Hilbert analysis}\label{se:RHP}
We can now implement the Riemann-Hilbert method to obtain the strong
asymptotics for the multiple Laguerre polynomials introduced in
Section \ref{se:MOP} and use it to prove Theorem \ref{thm:main2}.
The analysis is very similar to those in \cite{BKext1} (See also
\cite{Lysov}).

Let $C(f)$ be the Cauchy transform of the function $f(z)\in
L^2(\mathbb{R}_+)$ in $\mathbb{R}_+$
\begin{equation}\label{eq:cauchy}
C(f)(z)=\frac{1}{2\pi i}\int_{\mathbb{R}_+}\frac{f(s)}{s-z}ds,
\end{equation}
and let $w_1(z)$ and $w_2(z)$ be the weights of the multiple
Laguerre polynomials.
\begin{equation}\label{eq:weight}
w_1(z)=z^{M-N}e^{-Mz},\quad w_2(z)=z^{M-N}e^{-Ma^{-1}z},
\end{equation}
Denote by $\kappa_1$ and $\kappa_2$ the constants
\begin{equation*}
\kappa_1=-2\pi
i\left(h^{(1)}_{N_0-1,N_1}\right)^{-1},\quad\kappa_2=-2\pi
i\left(h^{(2)}_{N_0,N_1-1}\right)^{-1}.
\end{equation*}
Then due to the orthogonality condition (\ref{eq:multiop}), the
following matrix
\begin{equation}\label{eq:Ymatr}
Y(z)=\begin{pmatrix}P_{N_0,N_1}(z)&C(P_{N_0,N_1}w_1)(z)&C(P_{N_0,N_1}w_2)(z)\\
\kappa_1P_{N_0-1,N_1}(z)&\kappa_1C(P_{N_0-1,N_1}w_1)(z)&\kappa_1C(P_{N_0-1,N_1}w_2)(z)\\
\kappa_2P_{N_0,N_1-1}(z)&\kappa_2C(P_{N_0,N_1-1}w_1)(z)&\kappa_2C(P_{N_0,N_1-1}w_2)(z)
\end{pmatrix}
\end{equation}
is the unique solution of the following Riemann-Hilbert problem.
\begin{equation}\label{eq:RHPY}
\begin{split}
1.\quad &\text{$Y(z)$ is analytic in
$\mathbb{C}\setminus\mathbb{R}_+$},\\
2.\quad &Y_+(z)=Y_-(z)\begin{pmatrix}1&w_1(z)&w_2(z)\\
0&1&0\\
0&0&1
\end{pmatrix},\quad z\in\mathbb{R}_+\\
3.\quad &Y(z)=\left(I+O(z^{-1})\right)\begin{pmatrix}z^{N}&0&0\\
0&z^{-N_0}&0\\
0&0&z^{-N_1}
\end{pmatrix},\quad z\rightarrow\infty,\\
4.\quad  &Y(z)=O(1),\quad z\rightarrow 0.
\end{split}
\end{equation}
On the other hand, the following matrix
\begin{equation*}
X(z)=\begin{pmatrix}
-2\pi iC\left(A^1_{N_0,N_1}w_1+A^a_{N_0,N_1}w_2\right)&2\pi i A^1_{N_0,N_1}&2\pi iA^a_{N_0,N_1}\\
-h_{N_0,N_1}^{(1)}C\left(A^1_{N_0+1,N_1}w_1+A^a_{N_0+1,N_1}w_2\right)&h_{N_0,N_1}^{(1)} A^1_{N_0+1,N_1}&h_{N_0,N_1}^{(1)}A^a_{N_0+1,N_1}\\
-h_{N_0,N_1}^{(2)}C\left(A^1_{N_0,N_1+1}w_1+A^a_{N_0,N_1+1}w_2\right)&h_{N_0,N_1}^{(2)}
A^1_{N_0,N_1+1}&h_{N_0,N_1}^{(2)}A^a_{N_0,N_1+1}
\end{pmatrix}
\end{equation*}
is the unique solution of the following Riemann-Hilbert problem
\begin{equation}\label{eq:RHPX}
\begin{split}
1.\quad &\text{$X(z)$ is analytic in
$\mathbb{C}\setminus\mathbb{R}_+$},\\
2.\quad &X_+(z)=X_-(z)\begin{pmatrix}1&0&0\\
-w_1(z)&1&0\\
-w_2(z)&0&1
\end{pmatrix},\quad z\in\mathbb{R}_+\\
3.\quad &X(z)=\left(I+O(z^{-1})\right)\begin{pmatrix}z^{-N}&0&0\\
0&z^{N_0}&0\\
0&0&z^{N_1}
\end{pmatrix},\quad z\rightarrow\infty,\\
4.\quad  &X(z)=O(1),\quad z\rightarrow 0.
\end{split}
\end{equation}
In particular, by the uniqueness of the solution of the
Riemann-Hilbert problems (\ref{eq:RHPX}) and (\ref{eq:RHPY}), we see
that
\begin{equation}
X(z)=Y^{-T}(z).
\end{equation}
The proof of these statement can be found in \cite{vanGerKuij}. By a
similar computation as that in \cite{BKMOP} and \cite{BK1}, we can
express the kernel (\ref{eq:ker}) in terms of the solution of the
Riemann-Hilbert problem $Y(z)$.
\begin{equation}\label{eq:kerRHP}
\begin{split}
K_{M,N}(x,y)&=\frac{(xy)^{\frac{M-N}{2}}\left(e^{-My}\left[Y^{-1}_+(y)Y_+(x)\right]_{21}
+e^{-Ma^{-1}y}\left[Y^{-1}_+(y)Y_+(x)\right]_{31}\right)}{2\pi
i(x-y)},\\
&=\frac{(xy)^{\frac{M-N}{2}}}{2\pi i(x-y)}\left(0\quad e^{-My}\quad
e^{-Ma^{-1}y}\right)Y_+^{-1}(y)Y_+(x)\begin{pmatrix} 1 \\ 0\\0
\end{pmatrix}
\end{split}
\end{equation}
where $A_{21}$ and $A_{31}$ are the $21^{th}$ and $31^{th}$ entries
of $A$.
\subsection{First transformation of the Riemann-Hilbert problem}
We should now use the functions $\theta_j(z)$ constructed in Section
\ref{se:branch} to deform the Riemann-Hilbert problem
(\ref{eq:RHPY}). Our goal is to deform the Riemann-Hilbert problem
so that it can be approximated by a Riemann-Hilbert problem that is
explicitly solvable. Before we deform the Riemann-Hilbert problem,
let us make the following observation.

As pointed out in \cite{Baikspike}, Lemma \ref{le:CS} applies to any
distribution $G$ whose Stieltjes transform $m_{G}(z)$
(\ref{eq:stie}) satisfies an equation of the form
\begin{equation*}
z(m_{G})=-\frac{1}{m_{G}}+c_G\int_{\mathbb{R}}\frac{t}{1+tm_{G}}dH_G(t).
\end{equation*}
for some constant $c_G$ and distribution $H_G(t)$. Hence if we
replace the constants $c$ by $c_N=\frac{N}{M}$, $\beta$ by
$\beta_N=\frac{N_1}{N}$ and $H(t)$ by
$H_N(t)=(1-\beta_N)\delta_1+\beta_N\delta_{a}$, the results we
obtained in Section \ref{se:Stie} remain valid. In particular, since
the discriminant $\Delta$ in Theorem \ref{thm:density} is continuous
in the parameters $c$ and $\beta$, if $\Delta>0$, there exists large
enough $M$, $N$ and $N_1$ such that the discriminant $\Delta_N$
evaluated with the parameters $c_N$ and $\beta_N$ is positive. Let
$B_1^N$ and $B_2^N$ the following constants
\begin{equation*}
B_2^N=a(1-c_N),\quad B_1^N=1-c_N(1-\beta_N)+a(1-c_N\beta_N).
\end{equation*}
Then for large enough $M$, $N$ and $N_1$, the equation
\begin{equation}\label{eq:curveN}
za\xi^3+(A_2z+B_2^N)\xi^2+(z+B_1^N)\xi+1=0,
\end{equation}
will define a Riemann surface with 4 distinct real branch points
$0<\lambda_1^N<\ldots<\lambda_4^N$. In particular, let the roots of
this equation be $\xi_j^N(z)$ and define $\theta_j^N(z)$ as in
(\ref{eq:theta})
\begin{equation}\label{eq:thetaN}
\begin{split}
\theta_1^N(z)&=\int_{\lambda_4^N}^z\xi_1^N(x)dx,\quad
\theta_2^N(z)=\int_{\lambda_{k_2}^N}^z\xi_2^N(x)dx,\quad
\theta_3^N(z)=\int_{\lambda_{k_3}^N}^z\xi_3^N(x)dx,\\
k_2&=2,\quad k_3=4,\quad a>1,\quad k_2=4,\quad k_3=2,\quad a<1.
\end{split}
\end{equation}
Then all the results we obtained in Section \ref{se:branch} will
apply to the $\theta_j^N(z)$ and $\xi_j^N(z)$ with $c$ and $\beta$
replaced by $c_N$ and $\beta_N$. In particular, $\theta_j^N(z)$ has
the behavior near $z\rightarrow\infty$ and $z=0$ as indicated in
(\ref{eq:asymtheta}) with $c$, $\beta$ replaced by $c_N$ and
$\beta_N$ and $l_j$ replaced by some constants $l_j^N$.

We will now start deforming the Riemann-Hilbert problem
(\ref{eq:RHPY}). First let us define the functions $g_j^N(z)$ to be
\begin{equation}\label{eq:g}
\begin{split}
g_1^N(z)&=\theta_1^N(z)+(1-c_N)\log z,\quad
g_2^N(z)=\theta_2^N(z)+\theta_{1,-}^N(\lambda^N_{k_2})+z,\\
g_3^N(z)&=\theta_3^N(z)+\theta_{1,-}^N(\lambda^N_{k_3})+\frac{z}{a}.
\end{split}
\end{equation}
where the branch cut of $\log z$ in $g_1^N(z)$ is chosen to be the
negative real axis.

We then define $T(z)$ to be
\begin{equation}\label{eq:Tz}
T(z)=diag\left(e^{-Ml_1^N},e^{-Ml_2^N},e^{-Ml_3^N}\right)Y(z)
diag\left(e^{Mg_1^N(z)},e^{Mg_2^N(z)},e^{Mg_3^N(z)}\right)
\end{equation}
The matrix $T(z)$ will satisfy the following Riemann-Hilbert
problem.
\begin{equation}\label{eq:RHPT}
\begin{split}
1.\quad &\text{$T(z)$ is analytic in
$\mathbb{C}\setminus\mathbb{R}$},\\
2.\quad &T_+(z)=T_-(z)J_T(z),\quad z\in\mathbb{R},\\
3.\quad &T(z)=I+O(z^{-1}),\quad z\rightarrow\infty,\\
4.\quad  & T(z)=O(1),\quad z\rightarrow 0.
\end{split}
\end{equation}
where $J_T(z)$ is the following matrix
\begin{equation}\label{eq:JT}
\begin{split}
J_T(z)&=\begin{pmatrix}e^{M\left(\theta_{1,+}^N(z)-\theta_{1,-}^N(z)\right)}&
e^{M\left(\tilde{\theta}_{2,+}^N(z)-\theta_{1,-}^N(z)\right)}&
e^{M\left(\tilde{\theta}_{3,+}^N(z)-\theta_{1,-}^N(z)\right)}\\
0&e^{M\left(\theta_{2,+}^N(z)-\theta_{2,-}^N(z)\right)}&0\\
0&0&e^{M\left(\theta_{3,+}^N(z)-\theta_{3,-}^N(z)\right)}
\end{pmatrix},\\
\tilde{\theta}_{j}^N(z)&=\theta_j^N(z)+\theta_{1,-}^N(\lambda^N_{k_j}).
\end{split}
\end{equation}
By applying Lemma \ref{le:cuttheta} to the $\theta_j^N(z)$, we can
simplify the jump matrix $J_T(z)$. In particular, on
$[\lambda^N_{k_2-1},\lambda^N_{k_2}]$, we have
\begin{equation}\label{eq:JTk2}
\begin{split}
J_T(z)&=\begin{pmatrix}e^{M\left(\theta_1^N(z)-\tilde{\theta}_2^N(z)\right)_+}&
1&
e^{M\left(\tilde{\theta}_{3,+}^N(z)-\theta_{1,-}^N(z)\right)}\\
0&e^{M\left(\theta_1^N(z)-\tilde{\theta}_2^N(z)\right)_-}&0\\
0&0&1
\end{pmatrix},
\end{split}
\end{equation}
while on $[\lambda^N_{k_3-1},\lambda^N_{k_3}]$, we have
\begin{equation}\label{eq:JTk3}
\begin{split}
J_T(z)&=\begin{pmatrix}e^{M\left(\theta_1^N(z)-\tilde{\theta}_3^N(z)\right)_+}&
e^{M\left(\tilde{\theta}_{2,+}^N(z)-\theta_{1,-}^N(z)\right)}&
1\\
0&1&0\\
0&0&e^{M\left(\theta_1^N(z)-\tilde{\theta}_3^N(z)\right)_-}
\end{pmatrix},
\end{split}
\end{equation}
On the rest of the positive real axis, the jump matrix becomes
\begin{equation}\label{eq:JTrest}
\begin{split}
J_T(z)&=\begin{pmatrix}1&
e^{M\left(\tilde{\theta}_{2,+}^N(z)-\theta_{1,-}^N(z)\right)}&
e^{M\left(\tilde{\theta}_{3,+}^N(z)-\theta_{1,-}^N(z)\right)}\\
0&1&0\\
0&0&1
\end{pmatrix}.
\end{split}
\end{equation}
This is because $Mc_N\beta_N=N_1$ and $Mc_N(1-\beta_N)=N_0$ are both
integers. And on the negative real axis, the matrix $T(z)$ has no
jump for the same reason. Note that the jump matrix $J_T(z)$ is
continuous at $z=0$ as the off-diagonal entries of (\ref{eq:JTrest})
contain the factor $e^{-M\theta_{1,-}^N(z)}$ which vanishes at the
origin.

The Riemann-Hilbert problem for $T(z)$ now takes same form as the
one in \cite{BKext1} (See also \cite{Lysov}) and the techniques
developed there can now be applied to our problem.

\subsection{Lens opening and approximation of the Riemann-Hilbert
problem}

We will now apply the lens opening technique that is standard in the
Riemann-Hilbert analysis. (See e.g. \cite{BI}, \cite{D}, \cite{DKV},
\cite{DKV2}, \cite{BKext1}, \cite{Lysov}) Let us define the lens
contours $\Xi_{\pm}^{j}$ around a branch cut
$[\lambda^N_{k_j-1},\lambda^N_{k_j}]$ as in Figure \ref{fig:lens}.
\begin{figure}
\centering \psfrag{l1}[][][1][0.0]{\small$\lambda_{k_j-1}$}
\psfrag{l2}[][][1][0.0]{\small$\lambda_{k_j}$}
\psfrag{X+}[][][1][0.0]{\small$\Xi_+^j$}
\psfrag{X-}[][][1][0.0]{\small$\Xi_-^j$}
\includegraphics[scale=0.75]{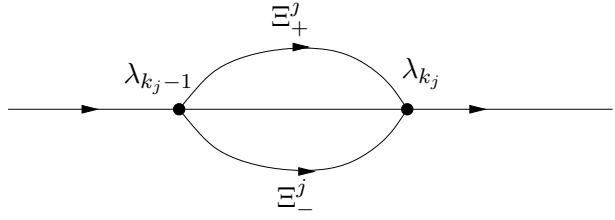}
\caption{The lens contours around a branch cut.}\label{fig:lens}
\end{figure}
We will chose the lens contours such that they are contained inside
the neighborhood $U_j$ stated in Lemma \ref{le:lens}. Now note that
the jump matrix $J_T(z)$ in (\ref{eq:JTk2}) and (\ref{eq:JTk3}) has
the following factorizations.
\begin{equation}\label{eq:Tfact}
\begin{split}
J_T(z)&=L_{j,-}(z)J_S(z)K_{j,+}(z),\quad
z\in[\lambda_{k_j-1}^N,\lambda_{k_j}^N],\\
J_S(z)&=\begin{pmatrix}0&1&0\\
-1&0&0\\
0&0&1
\end{pmatrix},\quad z\in[\lambda_{k_2-1}^N,\lambda_{k_2}^N],\\
J_S(z)&=\begin{pmatrix}0&0&1\\
0&1&0\\
-1&0&0
\end{pmatrix},\quad z\in[\lambda_{k_3-1}^N,\lambda_{k_3}^N],\\
\end{split}
\end{equation}
where $L_{j}(z)$, $K_{j}(z)$ are the following matrices
\begin{equation}\label{eq:Lj}
\begin{split}
L_2(z)&=\begin{pmatrix}1&0&0\\
e^{M\left(\theta_1^N(z)-\tilde{\theta}_2^N(z)\right)}&
1&-e^{M\left(\tilde{\theta}_3^N(z)-\tilde{\theta}_2^N(z)\right)}\\
0&0&1
\end{pmatrix},\\
K_2(z)&=\begin{pmatrix}1&0&0\\
e^{M\left(\theta_1^N(z)-\tilde{\theta}_2^N(z)\right)}&
1&e^{M\left(\tilde{\theta}_3^N(z)-\tilde{\theta}_2^N(z)\right)}\\
0&0&1
\end{pmatrix},\\
L_3(z)&=\begin{pmatrix}1&0&0\\
0&1&0\\
e^{M\left(\theta_1^N(z)-\tilde{\theta}_3^N(z)\right)}&-e^{M\left(\tilde{\theta}_2^N(z)-\tilde{\theta}_3^N(z)\right)}&1
\end{pmatrix},\\
K_3(z)&=\begin{pmatrix}1&0&0\\
0&1&0\\
e^{M\left(\theta_1^N(z)-\tilde{\theta}_3^N(z)\right)}&e^{M\left(\tilde{\theta}_2^N(z)-\tilde{\theta}_3^N(z)\right)}&1
\end{pmatrix}
\end{split}
\end{equation}
and the $\pm$ indices in (\ref{eq:Tfact}) are the boundary values of
the matrices $L_j(z)$ and $K_j(z)$ on the branch cuts.

If we define the matrix $S(z)$ to be
\begin{equation}\label{eq:Sz}
S(z)=\left\{
       \begin{array}{ll}
         T(z), & \hbox{$z$ outside of the lens regions;} \\
         T(z)K_j^{-1}(z), & \hbox{$z$ in the upper lens region of $[\lambda^N_{k_j-1},\lambda^N_{k_j}]$;} \\
         T(z)L_j(z), & \hbox{$z$ in the lower lens region of $[\lambda^N_{k_j-1},\lambda^N_{k_j}]$.}
       \end{array}
     \right.
\end{equation}
Then by the factorization (\ref{eq:Tfact}) and the jump of $T(z)$
(\ref{eq:JTrest}), we see that the matrix $S(z)$ is a solution to
the following Riemann-Hilbert problem.
\begin{equation}\label{eq:RHPS}
\begin{split}
1.\quad &\text{$S(z)$ is analytic in
$\mathbb{C}\setminus\left(\mathbb{R}_+\cup\Xi_{\pm}^j\right)$},\\
2.\quad &S_+(z)=S_-(z)J_S(z),\quad z\in\left(\mathbb{R}_+\cup\Xi_{\pm}^j\right),\\
3.\quad &S(z)=I+O(z^{-1}),\quad z\rightarrow\infty,\\
4.\quad  & S(z)=O(1),\quad z\rightarrow 0.
\end{split}
\end{equation}
where the matrix $J_S(z)$ is given by (\ref{eq:Tfact}) on
$[\lambda_{k_j-1},\lambda_{k_j}]$ and on $\Xi_{\pm}^j$, it is given
by
\begin{equation}\label{eq:JS1}
\begin{split}
J_S(z)&=K_j(z),\quad z\in\Xi_{+}^j,\quad J_S(z)=L_j(z),\quad
z\in\Xi_{-}^j.
\end{split}
\end{equation}
On
$\mathbb{R}_+\setminus\left(\cup_{j=1}^2[\lambda_{k_j-1}^N,\lambda_{k_j}^N]\right)$,
we have $J_S(z)=J_T(z)$.

Then by Lemma \ref{le:size} and Lemma \ref{le:lens}, we see that,
away from $[\lambda_{k_j-1}^N,\lambda_{k_j}^N]$ and from some small
neighborhoods $D_j$ of $\lambda_j^N$, the off-diagonal entries of
$J_S(z)$ are exponentially small as $M\rightarrow\infty$. This
suggests the following approximation to the Riemann-Hilbert problem
(\ref{eq:RHPS}).
\begin{equation}\label{eq:RHPSinf}
\begin{split}
1.\quad &\text{$S^{\infty}(z)$ is analytic in
$\mathbb{C}\setminus\left(\cup_{j=1}^2[\lambda_{k_j-1}^N,\lambda_{k_j}^N]\right)$},\\
2.\quad &S_+^{\infty}(z)=S_-^{\infty}(z)J_S(z),\quad z\in\cup_{j=1}^2[\lambda_{k_j-1}^N,\lambda_{k_j}^N],\\
3.\quad &S^{\infty}(z)=I+O(z^{-1}),\quad z\rightarrow\infty.
\end{split}
\end{equation}
In the next section we will give an explicit solution to this
Riemann-Hilbert problem and we will eventually show that
$S^{\infty}(z)$ is a good approximation of $S(z)$ when $z$ is
outside of the small neighborhoods $D_j$ of the branch points
$\lambda_j^N$.
\subsection{Outer parametrix}\label{se:outer}
The construction of the outer parametrix $S^{\infty}(z)$ is similar
to the one in \cite{BKext1} (See also \cite{Lysov}).

Let $\Lie^N$ be the Riemann surface defined by (\ref{eq:curveN}) and
let $\Gamma_j$ be the images of
$[\lambda_{k_j-1}^N,\lambda_{k_j}^N]$ on $\Lie^N$ under the map
$\xi_{1,+}^N(z)$. That is
\begin{equation}\label{eq:Gammaj}
\Gamma_j=\left\{(z,\xi)\in\Lie^N|\quad \xi=\xi_{1,+}^N(z),\quad
z\in[\lambda_{k_j-1}^N,\lambda_{k_j}^N].\right\},\quad j=2,3.
\end{equation}
Let us now define the functions $S^{\infty}_k(\xi)$, $k=1,2,3$ to be
the following functions on $\Lie^N$.
\begin{equation}\label{eq:Sinfk}
\begin{split}
S^{\infty}_1(\xi)&=a\sqrt{\prod_{j=1}^4\gamma_j^N}\frac{(\xi+1)(\xi+a^{-1})}{\sqrt{\prod_{j=1}^4(\xi-\gamma_j^N)}},\\
S^{\infty}_2(\xi)&=\frac{a\sqrt{\prod_{j=1}^4(1+\gamma_j^N)}}{a-1}\frac{\xi(\xi+a^{-1})}{\sqrt{\prod_{j=1}^4(\xi-\gamma_j^N)}},\\
S^{\infty}_3(\xi)&=\frac{\sqrt{\prod_{j=1}^4(1+a\gamma_j^N)}}{1-a}\frac{\xi(\xi+1)}{\sqrt{\prod_{j=1}^4(\xi-\gamma_j^N)}}.
\end{split}
\end{equation}
where $\gamma_j^N$ are the roots of polynomial
\begin{equation}\label{eq:quarticN}
\begin{split}
a^2(1-c_N)\xi^4
&+2(a^2(1-c_N\beta_N)+a(1-c_N(1-\beta_N))\xi^3\\
&+(1-c_N(1-\beta_N)+a^2(1-c_N\beta_N)+4a)\xi^2 +2(1+a)\xi+1\\
&=a^2(1-c_N)\prod_{j=1}^4(\xi-\gamma_j^N).
\end{split}
\end{equation}
The branch cuts of the square root in (\ref{eq:Sinfk}) are chosen to
be the contours $\Gamma_j$ (\ref{eq:Gammaj}) that joins $\gamma_1^N$
to $\gamma_2^N$ and $\gamma_3^N$ to $\gamma_4^N$.

By using the asymptotic behavior of the functions $\xi_j^N(z)$
(\ref{eq:xiinfty}) and (\ref{eq:zeroasym}), with $c$ and $\beta$
replaced by $c_N$ and $\beta_N$, we see that the all the functions
$S^{\infty}_k(\xi)$ are holomorphic near $\xi_j^N(0)$ for $j=1,2$
and $3$. Moreover, at the points $\xi_j^N(\infty)$, these functions
satisfy
\begin{equation}\label{eq:sasym}
\begin{split}
S^{\infty}_k(\xi_j^N(\infty))=\delta_{jk},\quad k,j=1,2,3.
\end{split}
\end{equation}
Let us define $S^{\infty}(z)$ to be the following matrix on
$z\in\mathbb{C}$.
\begin{equation}\label{eq:sinf}
\left(S^{\infty}(z)\right)_{ij}=S^{\infty}_i(\xi_j^N(z)),\quad 1\leq
i,j\leq 3.
\end{equation}
Then, since the branch cut of the square root in (\ref{eq:Sinfk})
are chosen to be $\Gamma_j$, we have, from the jump discontinuities
of the $\xi_j^N(z)$ (\ref{eq:bound}), the following
\begin{equation}\label{eq:Sjump}
S^{\infty}_i(\xi_{1,+}^N(z))=-S^{\infty}_i(\xi_{j,-}^N(z)),\quad
z\in[\lambda_{k_j-1},\lambda_{k_j}],\quad i=1,2,3,\quad j=2,3.
\end{equation}
From this and the asymptotic behavior (\ref{eq:sasym}) of the
$S^{\infty}_i(\xi)$, we see that the matrix $S^{\infty}(z)$
satisfies the Riemann-Hilbert problem (\ref{eq:RHPS}).
\begin{proposition}\label{pro:outer}
The matrix $S^{\infty}(z)$ defined by (\ref{eq:sinf}) satisfies the
following Riemann-Hilbert problem.
\begin{equation}\label{eq:outer}
\begin{split}
1.\quad &\text{$S^{\infty}(z)$ is analytic in
$\mathbb{C}\setminus\left(\cup_{j=1}^2[\lambda_{k_j-1}^N,\lambda_{k_j}^N]\right)$},\\
2.\quad &S_+^{\infty}(z)=S_-^{\infty}(z)J_S(z),\quad z\in\cup_{j=1}^2[\lambda_{k_j-1}^N,\lambda_{k_j}^N],\\
3.\quad &S^{\infty}(z)=I+O(z^{-1}),\quad z\rightarrow\infty,\\
4.\quad
&S^{\infty}(z)=O\left((z-\lambda_j^N)^{-\frac{1}{4}}\right),\quad
z\rightarrow\lambda_j^N, \quad j=1,\ldots, 4.
\end{split}
\end{equation}
where $J_S(z)$ is defined as in (\ref{eq:Tfact}).
\end{proposition}
\begin{proof} As we have already verified the jumps and the
asymptotic behavior at $z=\infty$, we only need to verify that
properties 1. and 4. in (\ref{eq:outer}) are true. From
(\ref{eq:Sinfk}), we see that as a function in $z$, the matrix
(\ref{eq:sinf}) has only jump discontinuities along the branch cuts
$[\lambda_{k_j-1}^N,\lambda_{k_j}^N]$ of $\xi_j^N$. Also, it can
only have singularities at the points $\gamma_k^N$ in which the
denominator vanishes, or at $z=0$ in which $\xi_1^N(z)$ has a pole.
As we have already pointed out that all the functions
$S^{\infty}_j(\xi)$ are holomorphic near the point $z=0$, the only
possible singularities are the points $z=\lambda_j^N$. Near these
points, 2 of the functions $\xi_j^N(z)$ will behave as
\begin{equation*}
\xi_j^N(z)=\gamma_k^N+C_j(z-\lambda_k^N)^{\frac{1}{2}}+O\left((z-\lambda_k^N)\right),\quad
z\rightarrow\lambda_k^N,
\end{equation*}
for some constants $C_j$. Hence the matrix (\ref{eq:sinf}) will have
a fourth root singularity at these points. This completes the proof
of the proposition.
\end{proof}
\subsection{Local parametrices near the edge points
$\lambda_k^N$}\label{se:local}

Near the edge points $\lambda_k^N$, the approximation of $S(z)$ by
$S^{\infty}(z)$ failed and we must solve the Riemann-Hilbert problem
exactly near these points and match the solutions to the outer
parametrix (\ref{eq:sinf}) up to an error term of order $O(M^{-1})$.
To be precise, let $\delta>0$ and let $D_k$ be a disc of radius
$\delta$ centered at the point $\lambda_k^N$, $k=1,\ldots,4$. We
would like to construct local parametrices $S^k(z)$ in $D_k$ such
that
\begin{equation}\label{eq:localpara}
\begin{split}
1.\quad &\text{$S^{k}(z)$ is analytic in
$D_k\setminus \left(\mathbb{R}\cup\Xi^k_-\cup\Xi^k_+\right)$},\\
2.\quad &S_+^{k}(z)=S_-^{k}(z)J_S(z),\quad z\in D_k\cap\left(\mathbb{R}\cup\Xi^k_-\cup\Xi^k_+\right),\\
3.\quad &S^{k}(z)=\left(I+O(M^{-1})\right)S^{\infty}(z),\quad z\in\p
D_k.
\end{split}
\end{equation}
The local parametrices $S^k(z)$ can be constructed by using the Airy
function as in \cite{BKext1} (See also \cite{Lysov}). Since the
construction is identical to that in \cite{BKext1} and \cite{Lysov},
we shall not go into the details but merely set up the notations and
state the results in \cite{BKext1} and \cite{Lysov}.

First recall that the Airy function $\mathrm{Ai}(z)$ is the unique
solution to the equation $v^{\prime\prime}=zv$ with the asymptotic
behavior given by (\ref{eq:asymairy}).

Let $\omega=e^{\frac{2\pi i}{3}}$, then functions $\mathrm{Ai}(z)$,
$\mathrm{Ai}(\omega z)$ and $\mathrm{Ai}(\omega^2 z)$ satisfy the
following linear relation,
\begin{equation}
\mathrm{Ai}(z)+\omega \mathrm{Ai}(\omega
z)+\omega^2\mathrm{Ai}(\omega^2 z)=0.
\end{equation}

Now note that, in the neighborhoods $D_{k_j}$ and $D_{k_j-1}$,
$j=2,3$, the functions
\begin{equation}\label{eq:fLR}
\begin{split}
f_j^R(z)&=\theta_1^N(z)-\theta_{1,-}(\lambda_{k_j}^N)-\theta_j^N(z)+\theta_{j,-}(\lambda_{k_j}^N),\\
f_j^L(z)&=\theta_1^N(z)-\theta_{1,-}(\lambda_{k_j-1}^N)-\theta_j^N(z)+\theta_{j,-}(\lambda_{k_j-1}^N)
\end{split}
\end{equation}
vanish like $\left(z-\lambda_{k_j}^N\right)^{\frac{3}{2}}$ and
$\left(z-\lambda_{k_j-1}^N\right)^{\frac{3}{2}}$ respectively, as
$z$ approaches $\lambda_{k_j}^N$ or $\lambda_{k_j-1}^N$ from the
lower half plane. Let us take the restriction of $f_j^R(z)$ and
$f_j^L(z)$ to the lower half plane, and analytically continue them
to $D_{k_j}\setminus[\lambda_{k_j}^N,\lambda_{k_j-1}^N]$ or to
$D_{k_j-1}\setminus[\lambda_{k_j}^N,\lambda_{k_j-1}^N]$. We will
denote these analytic continuation by $f_{j,l}^R(z)$ and
$f_{j,l}^L(z)$. Then the functions $f_{j,l}^R(z)$ and $f_{j,l}^L(z)$
will behave like $\left(z-\lambda_{k_j}^N\right)^{\frac{3}{2}}$ and
$\left(z-\lambda_{k_j-1}^N\right)^{\frac{3}{2}}$ in $D_{k_j}$ and
$D_{k_j-1}$ respectively. Therefore if we define the local function
$\zeta(z)$ by
\begin{equation}\label{eq:zeta}
\begin{split}
\zeta(z)&=\left(\frac{4}{3}f_{j,l}^R(z)\right)^{\frac{2}{3}},\quad z\in D_{k_j},\\
\zeta(z)&=\left(\frac{4}{3}f_{j,l}^L(z)\right)^{\frac{2}{3}},\quad
z\in D_{k_j-1},
\end{split}
\end{equation}
then $\zeta(z)$ will be holomorphic inside the neighborhoods $D_k$.
Moreover, for small enough $\delta>0$, $\zeta(z)$ will be conformal
inside $D_k$. (The $\frac{4}{3}$ factor is introduced to simplify
calculation) We should make use of the freedom in the definition of
the lens contours to deform them such that inside $D_2$ and $D_4$,
the upper and lower lens contours coincide with the contours
$\arg(\zeta(z))=\frac{2\pi}{3}$ and $\arg(\zeta(z))=-\frac{2\pi}{3}$
respectively, while inside $D_1$ and $D_3$, we will deform the lens
contours such that the upper and lower lens contours coincide with
$\arg(\zeta(z))=\frac{\pi}{3}$ and $\arg(\zeta(z))=-\frac{\pi}{3}$
respectively.

Let us now consider the matrix $\tilde{S}^k(z)$ that is related to
$S^k(z)$ by
\begin{equation}\label{eq:tildeSk}
\begin{split}
\tilde{S}^{k_2-i}(z)&=\left\{
                     \begin{array}{ll}
                       S^{k_2-i}(z), & \hbox{$z$ inside the lens of $D_{k_2-i}$;} \\
                       S^{k_2-i}(z)\begin{pmatrix}1&0&0\\
                                  0&1&-e^{M\left(\tilde{\theta}_3^N(z)-\tilde{\theta}_2^N(z)\right)}\\
                                   0&0&1\end{pmatrix}, & \hbox{$z$ outside the lens of $D_{k_2-i}$.}
                     \end{array}
                   \right.\\
\tilde{S}^{k_3-i}(z)&=\left\{
                     \begin{array}{ll}
                       S^{k_3-i}(z), & \hbox{$z$ inside the lens of $D_{k_3-i}$;} \\
                       S^{k_3-i}(z)\begin{pmatrix}1&0&0\\
                                  0&1&0\\
                                   0&-e^{M\left(\tilde{\theta}_2^N(z)-\tilde{\theta}_3^N(z)\right)}&1\end{pmatrix}, & \hbox{$z$ outside the lens of $D_{k_3-i}$.}
                     \end{array}
                   \right.
\end{split}
\end{equation}
where $i=0,1$ in the above. In \cite{BKext1} (See also
\cite{Lysov}), matrices $\tilde{S}^k(z)$ was constructed such that
when $S^k(z)$ is related to $\tilde{S}^k(z)$ by (\ref{eq:tildeSk}),
$S^k(z)$ will solve the problem (\ref{eq:localpara}). These matrices
have the following form (See also \cite{BI}, \cite{D}, \cite{DKV},
\cite{DKV2}).
\begin{equation}\label{eq:local}
\tilde{S}^k(z)=E_k(z)W_k\Phi_k\left(M^{\frac{2}{3}}\zeta(z)\right)Z_k,\quad
z\in D_k,
\end{equation}
where $Z_k$ is the diagonal matrix
\begin{equation}\label{eq:Zk}
\begin{split}
Z_k&=\mathrm{diag}\left(e^{\frac{1}{2}M\left(\theta_1^N(z)-\tilde{\theta}_2^N(z)\right)},e^{-\frac{1}{2}M\left(\theta_1^N(z)-\tilde{\theta}_2^N(z)\right)}
,1\right),\quad k=k_2,k_2-1,\\
Z_k&=\mathrm{diag}\left(e^{\frac{1}{2}M\left(\theta_1^N(z)-\tilde{\theta}_3^N(z)\right)},1,
e^{-\frac{1}{2}M\left(\theta_1^N(z)-\tilde{\theta}_3^N(z)\right)}
\right),\quad k=k_3,k_3-1.
\end{split}
\end{equation}
The matrix $E_k(z)$ is a holomorphic matrix inside $D_k$ that is
bounded as $M$, $N$ and $N_1\rightarrow\infty$. The matrix $W_k$ is
a constant invertible diagonal matrix of order $M^{\frac{1}{6}}$.
They are here to fix the boundary condition in (\ref{eq:localpara}).
The $\Phi_k(z)$ is a matrix whose entries consist of the different
branches of the Airy function.

For example, when $k=4$ and $a<1$, the matrices $E_k(z)$, $W_k$ and
$\Phi_k(z)$ are given by
\begin{equation*}
\begin{split}
\Phi_4(z)&=\left\{
            \begin{array}{ll}
              \begin{pmatrix}v_0(z)&-v_2(z)&0\\
                              v_0^{\prime}(z)&-v_2^{\prime}(z)&0\\
              0&0&1\end{pmatrix}, & \hbox{for $0<\arg(z)<\frac{2\pi}{3}$;} \\
              \begin{pmatrix}-v_1(z)&-v_2(z)&0\\
                              -v_1^{\prime}(z)&-v_2^{\prime}(z)&0\\
              0&0&1\end{pmatrix}, & \hbox{for $\frac{2\pi}{3}<\arg(z)<\pi$;} \\
              \begin{pmatrix}-v_2(z)&v_1(z)&0\\
                              -v_2^{\prime}(z)&v_1^{\prime}(z)&0\\
              0&0&1\end{pmatrix}, & \hbox{for $-\pi<\arg(z)<-\frac{2\pi}{3}$;} \\
              \begin{pmatrix}v_0(z)&v_1(z)&0\\
                              v_0^{\prime}(z)&v_1^{\prime}(z)&0\\
              0&0&1\end{pmatrix}, & \hbox{for $-\frac{2\pi}{3}<\arg(z)<0$.}
            \end{array}
          \right.\\
E_4(z)&=\sqrt{\pi}S^{\infty}(z)\begin{pmatrix}1&-1&0\\
                              -i&-i&0\\
              0&0&1\end{pmatrix}\begin{pmatrix}\zeta^{\frac{1}{4}}&0&0\\
                              0&\zeta^{-\frac{1}{4}}&0\\
              0&0&1\end{pmatrix},\\
W_4&=diag(M^{\frac{1}{6}},M^{-\frac{1}{6}},1).
\end{split}
\end{equation*}
where the functions $v_j(z)$ are given by the Airy functions
$v_0(z)=\mathrm{Ai}(z)$, $v_1(z)=\omega \mathrm{Ai}(\omega z)$ and
$v_2(z)=\omega^2\mathrm{Ai}(\omega^2z)$. Note that despite the
apparent singularity $\zeta^{-\frac{1}{4}}$ in the expression of
$E_4(z)$, it turns out that $E_4(z)$ is holomorphic inside $D_4$
because the factor $S^{\infty}(z)$ contains terms that would cancel
out this singularity.

\subsection{Last transformation of the Riemann-Hilbert problem}
Let us now show that the parametrices we constructed in Section
\ref{se:outer} and Section \ref{se:local} are indeed good
approximation to the solution $S(z)$ of the Riemann-Hilbert problem
(\ref{eq:RHPS}).

Let us define $R(z)$ to be the following matrix
\begin{equation}\label{eq:Rx}
\begin{split}
R(z)=\left\{
       \begin{array}{ll}
         S(z)\left(S^{k}(z)\right)^{-1}, & \hbox{$z$ inside $D_k$, $k=1,\ldots,4$;} \\
         S(z)\left(S^{\infty}(z)\right)^{-1}, & \hbox{$z$ outside of $D_k$, $k=1,\ldots,4$.}
       \end{array}
     \right.
\end{split}
\end{equation}
Then the function $R(z)$ has jump discontinuities on the contour
$\Gamma_R$ shown in Figure \ref{fig:sigma}.
\begin{figure}
\centering
\includegraphics[scale=0.75]{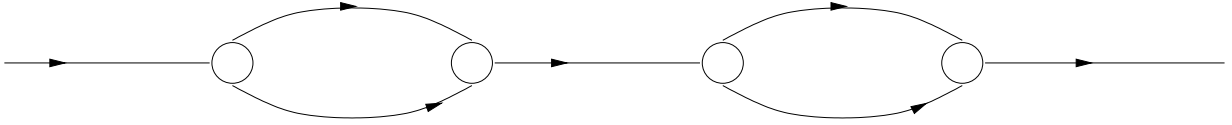}
\caption{The contour $\Gamma_R$.}\label{fig:sigma}
\end{figure}
In particular, $R(z)$ satisfies the Riemann-Hilbert problem
\begin{equation}\label{eq:RHR}
\begin{split}
&1. \quad \text{$R(z)$ is analytic in $\mathbb{C}\setminus\Gamma_R$}\\
&2.\quad R_+(z)=R_-(z)J_R(z)\\
&3. \quad R(z)=I+O(z^{-1}),\quad z\rightarrow\infty, \\
&4.\quad R(z)=O(1),\quad z\rightarrow 0.
\end{split}
\end{equation}
From the definition of $R(z)$ (\ref{eq:Rx}), it is easy to see that
the jumps $J_R(z)$ has the following order of magnitude.
\begin{equation}\label{eq:Jx}
\begin{split}
J_R(z)=\left\{
       \begin{array}{ll}
         I+O(M^{-1}), & \hbox{$z\in\p D_k$, $k=1,\ldots, 4$ ;} \\
         I+O\left(e^{-M\eta}\right), & \hbox{for some fixed $\eta>0$ on
the rest of $\Gamma_R$.}
       \end{array}
     \right.
\end{split}
\end{equation}
Then by the standard theory, \cite{D}, \cite{DKV}, \cite{DKV2}, we
have
\begin{equation}\label{eq:Rest}
\begin{split}
R(z)=I+O\left(\frac{1}{M(|z|+1)}\right),
\end{split}
\end{equation}
uniformly in $\mathbb{C}$.

In particular, the solution $S(z)$ of the Riemann-Hilbert problem
(\ref{eq:RHPS}) can be approximated by $S^{\infty}(z)$ and
$S^{k}(z)$ as
\begin{equation}\label{eq:approxS}
\begin{split}
S(z)=\left\{
       \begin{array}{ll}
         \left(I+O\left(M^{-1}\right)\right)S^{k}(z), & \hbox{$z\in D_k$, $k=1,\ldots,4$;} \\
         \left(I+O\left(M^{-1}\right)\right)S^{\infty}(z), & \hbox{$z$ outside of
$D_k$, $k=1,\ldots,4$.}
       \end{array}
     \right.
\end{split}
\end{equation}
\section{Universality of the correlation kernel}
We shall now use the asymptotics of the multiple Laguerre
polynomials obtained in the last section to prove the universality
results in Theorem \ref{thm:main2}. Since the proofs are the same as
the ones in Section 9 of \cite{BKext1}, (See also \cite{BI},
\cite{D}, \cite{DKV}, \cite{DKV2}) we shall only carry out the the
proof for (\ref{eq:bulk}) and leave the readers to verify
(\ref{eq:edge}).

First let us express the correlation kernel $\hat{K}_{M,N}(x,y)$ in
(\ref{eq:khat}) in terms of the matrix $S(x)$ in (\ref{eq:Sz}). By
(\ref{eq:kerRHP}), (\ref{eq:khat}), (\ref{eq:Tz}) and (\ref{eq:Sz}),
we have
\begin{equation}\label{eq:Sker2}
K_{M,N}(x,y)=\frac{\left(x^{-1}y\right)^{\frac{M-N}{2}}}{2\pi
i(x-y)}\left(-e^{M\theta_{1,+}^N(y)}\quad
e^{M\tilde{\theta}_{2,+}^N(y)}\quad
0\right)S_+^{-1}(y)S_+(x)\begin{pmatrix} e^{-M\theta_{1,+}^N(x)}
\\ e^{-M\tilde{\theta}_{2,+}^N(x)}\\0
\end{pmatrix},
\end{equation}
for $x,y\in[\lambda_{k_2-1}^N,\lambda_{k_2}^N]$ and
\begin{equation}\label{eq:Sker3}
K_{M,N}(x,y)=\frac{\left(x^{-1}y\right)^{\frac{M-N}{2}}}{2\pi
i(x-y)}\left(-e^{M\theta_{1,+}^N(y)}\quad 0\quad
e^{M\tilde{\theta}_{3,+}^N(y)}\right)S_+^{-1}(y)S_+(x)\begin{pmatrix}
e^{-M\theta_{1,+}^N(x)}
\\ 0\\e^{-M\tilde{\theta}_{3,+}^N(x)}
\end{pmatrix}
\end{equation}
for $x,y\in[\lambda_{k_3-1}^N,\lambda_{k_3}^N]$. Now note that,
since $\xi_1^N(z)$ and $\xi_j^N(z)$ are complex conjugates on
$[\lambda_{k_3-1}^N,\lambda_{k_3}^N]$, by (\ref{eq:thetaN}), we have
\begin{equation}\label{eq:conju1}
\theta_{1,+}^N(z)-\theta_{1,+}^N(\lambda_{k_j}^N)=\overline{\theta_{j,+}^N(z)},
\end{equation}
where $\overline{A}$ denotes the complex conjugation.

If $k_j=4$, then
$\theta_{1,+}^N(\lambda_{k_j}^N)=\theta_{1,-}(\lambda_{k_j}^N)=0$
and if $k_j=2$, then
\begin{equation*}
\begin{split}
\theta_{1,\pm}^N(\lambda_{k_j}^N)&=\int_{\lambda_{4}^N}^{\lambda_2^N}\xi_{1,\pm}^N(x)dx,\\
&=\int_{\lambda_{4}^N}^{\lambda_3^N}\xi_{1,\pm}^N(x)dx+\int_{\lambda_{3}^N}^{\lambda_2^N}\xi_{1,\pm}^N(x)dx.
\end{split}
\end{equation*}
Since $\xi_{1,+}^N(z)=\overline{\xi_{1,-}^N(z)}$ in
$[\lambda_3^N,\lambda_4^N]$ and
$\xi_{1,+}^N(z)=\xi_{1,-}^N(z)\in\mathbb{R}$ in
$[\lambda_2^N,\lambda_3^N]$, we see that, in either case of $k_j=2$
or $k_j=4$, we have
$\theta_{1,+}^N(\lambda_{k_j}^N)=\overline{\theta_{1,-}^N(\lambda_{k_j}^N)}$.
Hence we have, by (\ref{eq:conju1}) and (\ref{eq:JT}),
\begin{equation}\label{eq:conju}
\theta_{1,+}^N(z)=\overline{\tilde{\theta}_{j,+}^N(z)},\quad
z\in[\lambda_{k_j-1}^N,\lambda_{k_j}^N].
\end{equation}
By the same argument, same relations between $\theta_j(z)$ and
$\tilde{\theta}_j(z)$ can be obtained.

By substituting (\ref{eq:conju}) and (\ref{eq:khat}) into
(\ref{eq:Sker2}), we obtain
\begin{equation}\label{eq:Sker21}
\begin{split}
\hat{K}_{M,N}(x,y)&=\frac{e^{-M\left(h_N(x)-h_N(y)\right)}}{2\pi
i(x-y)}\left(-e^{M\mathrm{Im}\left(\theta_{1,+}^N(y)\right)}\quad
e^{-M\mathrm{Im}\left(\theta_{1,+}^N(y)\right)}\quad
0\right)S_+^{-1}(y)\\
&\times S_+(x)\begin{pmatrix} e^{-MIm\left(\theta_{1,+}^N(x)\right)}
\\ e^{MIm\left(\theta_{1,+}^N(x)\right)}\\0
\end{pmatrix},
\end{split}
\end{equation}
on $[\lambda_{k_2-1}^N,\lambda_{k_2}^N]$, where $h_N(x)$ is given by
\begin{equation*}
\begin{split}
h_N(x)&=\mathrm{Re}\left(\theta_{1,+}^N(x)\right)-\frac{1}{2}\left(\theta_{1,+}(x)+\theta_{j,+}(x)
+\theta_{1,-}(\lambda_{k_j})\right) \\
&=
\mathrm{Re}\left(\theta_{1,+}^N(x)-\theta_{1,+}(x)\right).
\end{split}
\end{equation*}
where we have used (\ref{eq:conju}) in the second equality. We can
obtain similar expression for
$x,y\in[\lambda_{k_3-1}^N,\lambda_{k_3}^N]$ by substituting
(\ref{eq:conju}) into (\ref{eq:Sker3}).

We can now apply (\ref{eq:approxS}) to (\ref{eq:Sker21}) and its
counterpart on $[\lambda_{k_3-1}^N,\lambda_{k_3}^N]$ to prove
(\ref{eq:bulk}) in Theorem \ref{thm:main2}.

Let us consider the points $x$ and
$y\in(\lambda_{k_j-1},\lambda_{k_j})$ such that
\begin{equation}\label{eq:xy}
x=x_0+\frac{u}{M\rho_N(x_0)},\quad y=x_0+\frac{v}{M\rho_N(x_0)},
\end{equation}
where $x_0\in(\lambda_{k_j-1},\lambda_{k_j})$. For large enough $N$,
we can assume that $x_0$, $x$ and $y$ are in
$(\lambda_{k_j-1}^N,\lambda_{k_j}^N)$ and outside of $D_{k_j-1}$ and
$D_{k_j}$. The function $\rho_N(z)$ in (\ref{eq:xy}) is given by
\begin{equation}\label{eq:rhoN}
\rho_N(z)=\frac{1}{\pi}\mathrm{Im}\left(\xi_{1,+}^N(z)\right),
\end{equation}
Let us replace the matrix $S_+^{-1}(y)S_+(x)$ in (\ref{eq:Sker21})
by
\begin{equation*}
S_+^{-1}(y)S_+(x)=\left(S_+^{\infty}(y)\right)^{-1}R^{-1}(y)R(x)S^{\infty}_+(x).
\end{equation*}
Since $R(z)$ is analytic at $x_0$, by (\ref{eq:Rest}) and by
considering the power series expansion of $R(z)$ at $x_0$, we see
that $R^{-1}(y)R(x)$ is of order $I+O\left(\frac{x-y}{M}\right)$. On
the other hand, since $x_0$ is outside of the neighborhoods $D_k$ of
the branch points $\lambda_k^N$, the matrix-valued function
$S_+^{\infty}(z)$ has an analytic continuation in a neighborhood of
$x_0$ that includes both $x$ and $y$. Hence by considering the power
series expansion of $S_+^{\infty}(z)$ at $x_0$, we have
\begin{equation}\label{eq:Sorder}
\begin{split}
S_+^{-1}(y)S_+(x)&=\left(S_+^{\infty}(y)\right)^{-1}R^{-1}(y)R(x)S^{\infty}_+(x)\\
&=\left(S_+^{\infty}(y)\right)^{-1}\left(I+O\left(\frac{x-y}{M}\right)\right)S^{\infty}_+(x),\\
&=I+O\left(x-y\right)
\end{split}
\end{equation}
By substituting (\ref{eq:Sorder}) back into (\ref{eq:Sker21}) and
its counter part on $[\lambda_{k_3-1}^N,\lambda_{k_3}^N]$, we obtain
\begin{equation}\label{eq:Sker22}
\hat{K}_{M,N}(x,y)=e^{-M\left(h_N(x)-h_N(y)\right)}\left(\frac{\sin\left(M\mathrm{Im}
\left(\theta_{1,+}^N(x)-\theta_{1,+}^N(y)\right)\right)}{\pi(x-y)}
+O(1)\right),
\end{equation}
for $x,y\in(\lambda_{k_j-1}^N,\lambda_{k_j}^N)$. Since
$\theta_{1,+}^N(z)$ is given by (\ref{eq:thetaN}) and that $x_0$ is
away from the branch points $\lambda_k^N$, the functions
$\theta_{1,+}^N(z)$ and $\xi_{1,+}^N(z)$ can be continued
analytically in a neighborhood of $x_0$. If we denote these
analytical continuation by the same symbols, we will have
\begin{equation}\label{eq:thetax0}
\theta_{1,+}^N(z)=\theta_{1,+}^N(x_0)+\xi_{1,+}^N(z-x_0)+O\left((z-x_0)^2\right),\quad
z\rightarrow x_0.
\end{equation}
Now by substituting (\ref{eq:xy}) and (\ref{eq:thetax0}) into
(\ref{eq:Sker22}), we obtain
\begin{equation}\label{eq:sine}
\hat{K}_{M,N}(x,y)=e^{-M\left(h_N(x)-h_N(y)\right)}\left(M\rho_N(x_0)\frac{\sin\left(\mathrm{Im}
\left(\frac{\xi_{1,+}^N(x_0)}{\rho_N(x_0)}\right)(u-v)\right)}{\pi(u-v)}
+O(1)\right).
\end{equation}
From the definition of $\rho_N(z)$ (\ref{eq:rhoN}), we see that
\begin{equation*}
\mathrm{Im} \left(\frac{\xi_{1,+}^N(x_0)}{\rho_N(x_0)}\right)=\pi
\end{equation*}
Now recall that the density $\rho(z)$ on the support of
$\underline{F}$ is given by (See Theorem \ref{thm:density})
\begin{equation*}
\rho(z)=\frac{1}{\pi}\mathrm{Im}\left(\xi_{1,+}(z)\right).
\end{equation*}
As the solutions of (\ref{eq:curve2}) are continuous in the
parameters $a$, $\beta$ and $c$, we have
\begin{equation}\label{eq:xirho}
\begin{split}
\lim_{N\rightarrow\infty}\xi_{1,+}^N(x_0)=\xi_{1,+}(x_0),\quad
\lim_{N\rightarrow\infty}\rho_N(x_0)=\rho(x_0).
\end{split}
\end{equation}
Let us now show that
\begin{equation}\label{eq:hlim}
\lim_{M,N\rightarrow\infty}e^{-M\left(h_N(x)-h_N(y)\right)}=1.
\end{equation}
By (\ref{eq:xy}) and (\ref{eq:thetax0}), we see that
\begin{equation*}
\lim_{M,N\rightarrow\infty}M\left(h_N(x)-h_N(y)\right)=\lim_{M,N\rightarrow\infty}\mathrm{Re}\left(\xi_{1,+}^N(x_0)-\xi_{1,+}(x_0)\right)\frac{u-v}{\rho_N(x_0)}.
\end{equation*}
Then by (\ref{eq:xirho}), we obtain (\ref{eq:hlim}).

Now if we replace $u$ and $v$ in (\ref{eq:sine}) by the real
variables $u\rightarrow\frac{u\rho_N(x_0)}{\rho(x_0)}$,
$v\rightarrow\frac{v\rho_N(x_0)}{\rho(x_0)}$ and take the limit, we
obtain
\begin{equation}\label{eq:sineker}
\lim_{N,M\rightarrow\infty}\frac{\hat{K}_{M,N}\left(x_0+\frac{u}{\rho(x_0)},x_0+\frac{v}{\rho(x_0)}\right)}{M}=\rho(x_0)\frac{\sin
\pi(u-v)}{\pi(u-v)}.
\end{equation}
This proves (\ref{eq:bulk}).

The proof for (\ref{eq:edge}) is the same as the one in Section 9 of
\cite{BKext1}. Let $x$ and $y$ be
\begin{equation}\label{eq:xyedge}
x=\lambda_k^N+\frac{u}{\left(\rho_k^NM\right)^{\frac{2}{3}}},\quad
y=\lambda_k^N+\frac{v}{\left(\rho_k^NM\right)^{\frac{2}{3}}}.
\end{equation}
Then following through the details in Section 9 of \cite{BKext1}, we
would arrive at
\begin{equation}\label{eq:edgeN}
\begin{split}
&\frac{e^{-M\left(h_N(x)-h_N(y)\right)}}{\left(M\rho_k^N\right)^{\frac{2}{3}}}
\hat{K}_{M,N}\left(\lambda_k^N+\frac{u}{\left(M\rho_k^N\right)^{\frac{2}{3}}},
\lambda_k^N+\frac{v}{\left(M\rho_k^N\right)^{\frac{2}{3}}}\right)=\\
&\frac{\mathrm{Ai}(u)\mathrm{Ai}^{\prime}(v)
-\mathrm{Ai}^{\prime}(u)\mathrm{Ai}(v)}{u-v}+O\left(M^{-\frac{1}{3}}\right),
\end{split}
\end{equation}
We will leave it to the readers to verify the details. To obtain
(\ref{eq:edge}) from (\ref{eq:edgeN}), we will need to use the
following.
\begin{lemma}\label{le:lambdak}
Under the limit (\ref{eq:lim}), $\lambda_k$ satisfies the following
\begin{equation*}
\lambda_k=\lambda_k^N+O\left(M^{-1}\right)
\end{equation*}
\end{lemma}
\begin{proof} Let $\alpha_{\epsilon}$ be a solution
of the following quartic equation
\begin{equation*}
\begin{split}
W_0(\alpha_{\epsilon})+\epsilon W_1(\alpha_{\epsilon}) =0.
\end{split}
\end{equation*}
where $W_j(x)$ are polynomials of degree $n-j$ whose coefficients
are bounded in $\epsilon$. Let $\alpha_0$ be the solution when
$\epsilon=0$. Suppose $\alpha_{\epsilon}=\alpha_{0}+\delta$. Then we
have
\begin{equation*}
\delta W_0^{\prime}(\alpha_0)+\epsilon
W_1(\alpha_0)+O(\delta^2)+O(\delta\epsilon)+O(\epsilon^2)=0.
\end{equation*}
Therefore if $\frac{\epsilon}{\delta}=o(1)$, we will have
\begin{equation*}
W_0^{\prime}(\alpha_0)=0.
\end{equation*}
This implies that $\alpha_0$ is a double root of $W_0(x)$. We can
apply this to the polynomial (\ref{eq:quartic1}) and
(\ref{eq:quarticN}). By (\ref{eq:lim}), we have
$c-c_N=O\left(M^{-1}\right)$ and
$\beta-\beta_N=O\left(M^{-1}\right)$. Therefore the differences
between (\ref{eq:quartic1}) and (\ref{eq:quarticN}) is a cubic
polynomial whose coefficients are of order $O\left(M^{-1}\right)$.

Since the polynomial (\ref{eq:quartic1}) has no double root, the
difference between $\lambda_k$ and $\lambda_k^N$ must be of order
$O\left(M^{-1}\right)$.
\end{proof}
Let us now show that (\ref{eq:hlim}) is true in this case. Near the
edge point $\lambda_k$, the function $h_N(x)-h_N(y)$ is given by
\begin{equation*}
\begin{split}
h_N(x)-h_N(y)&=\frac{2}{3}\left(\frac{u^{\frac{3}{2}}-v^{\frac{3}{2}}}{M}\right)-
\frac{2}{3}\rho_k\left(\left(\lambda_k^N-\lambda_k+\frac{u}{\left(\rho_k^NM\right)^{\frac{2}{3}}}\right)\right)^{\frac{3}{2}}
\\
&+\frac{2}{3}\rho_k\left(\left(\lambda_k^N-\lambda_k+\frac{v}{\left(\rho_k^NM\right)^{\frac{2}{3}}}\right)\right)^{\frac{3}{2}}
+O\left(M^{-\frac{4}{3}}\right).
\end{split}
\end{equation*}
Since $\lambda_k^N-\lambda_k=O\left(M^{-1}\right)$, the above
becomes
\begin{equation*}
h_N(x)-h_N(y)=\frac{2}{3}\left(1-\frac{\rho_k}{\rho_k^N}\right)
\left(\frac{u^{\frac{3}{2}}-v^{\frac{3}{2}}}{M}\right)+O\left(M^{-\frac{4}{3}}\right).
\end{equation*}
This implies that (\ref{eq:hlim}) is true also at the edge.

Now let us replace $u$ and $v$ in (\ref{eq:edgeN}) by
\begin{equation*}
\begin{split}
u&\rightarrow
\left(\frac{\rho_k^N}{\rho_k}\right)^{\frac{2}{3}}u-\left(\rho_k^NM\right)^{\frac{2}{3}}(\lambda_k^N-\lambda_k),\\
v&\rightarrow
\left(\frac{\rho_k^N}{\rho_k}\right)^{\frac{2}{3}}v-\left(\rho_k^NM\right)^{\frac{2}{3}}(\lambda_k^N-\lambda_k).
\end{split}
\end{equation*}
Since $\lambda_k^N-\lambda_k=O\left(M^{-1}\right)$, the new $u$ and
$v$ are still real and of the same order in $M$. By making these
replacement, and then take the limit of (\ref{eq:edgeN}), we arrive
at (\ref{eq:edge}).

\vspace{.25cm}

\noindent\rule{16.2cm}{.5pt}

\vspace{.25cm}

{\small

\noindent {\sl School of Mathematics \\
                       University of Bristol\\
                       Bristol BS8 1TW, UK  \\
                       Email: {\tt m.mo@bristol.ac.uk}

                       \vspace{.25cm}

                       \noindent  22 September  2008}}

\end{document}